\documentclass{amsart}

\usepackage{amsmath,amssymb}

\newcommand\largediamond{\mbox{\large$\diamondsuit$}}
\newcommand{\w}{\omega}
\newcommand{\mr}{\mathrm}

\newcommand{\mc}{\mathcal}
\newcommand{\bsigma}{\mathbf{\mathop{\pmb{\sum}}}^-}

\newcommand{\mathnot}{{\sim}}
\newcommand\forces{\Vdash}
\newcommand\val{\mathop{val}}
\newcommand\dom{\operatorname{dom}}
\def\aweebitback{\mskip-1.5\thinmuskip}
\def\supy#1{\aweebitback\vphantom{\vec{y}}^{#1}}

\newtheorem{theorem}{Theorem}[section]

\newtheorem{lem}[theorem]{Lemma}

\newtheorem{lemma}[theorem]{Lemma}
\newtheorem{corollary}[theorem]{Corollary}
\newtheorem{proposition}[theorem]{Proposition}
\newtheorem{defn}[theorem]{Definition}
\newtheorem{definition}[theorem]{Definition}

\newtheorem{claim}{Claim}

\newsavebox{\Prfref}

\newsavebox{\prfref}

\begin{document}

\title[Hereditarily Normal Manifolds of Dimension $>1$]{Hereditarily
  Normal Manifolds of Dimension~$>1$ May All Be Metrizable} 
\author{Alan Dow{$^1$} and Franklin D. Tall{$^2$}}

\footnotetext[1]{Research supported by NSF grant  DMS-1501506.}
\footnotetext[2]{Research supported by NSERC grant A-7354.\vspace*{2pt}}
\date{\today}
\maketitle

\begin{abstract}
  P.\ J.\ Nyikos has asked whether it is consistent that every
  hereditarily normal manifold of dimension $>1$ is metrizable, and
  proved it is if one assumes the consistency of a supercompact
  cardinal, and, in addition, that the manifolds are hereditarily
  collectionwise Hausdorff. We are able to omit these extra
  assumptions.  
\end{abstract}

\renewcommand{\thefootnote}{} \footnote
{\parbox[1.8em]{\linewidth}{$2000$ Math.\ Subj.\ Class.\ Primary
    54A35, 54D15, 54D45, 54E35, 03E05, 03E35, 03E65; Secondary 54D20,
    03E55.}\vspace*{5pt}} \renewcommand{\thefootnote}{} \footnote
{\parbox[1.8em]{\linewidth}{Key words and phrases: hereditarily
    normal, manifold, metrizable, coherent Souslin tree, proper
    forcing, $\mr{PFA}(S)[S]$, locally compact, $P$-ideal, perfect
    pre-image of $\w_1$, sequentially compact.}}

\section{Nyikos' Manifold Problem}

For us, a \emph{manifold} is simply a  locally Euclidean
topological space.  Mary Ellen Rudin proved that
$\mr{MA}+\mathnot\mr{CH}$ implies every perfectly normal manifold is
metrizable~\cite{Rudin1979}. Hereditary normality ($T_5$) is a natural
weakening of perfect normality; Peter Nyikos noticed that, although
the Long Line and Long Ray are hereditarily normal non-metrizable
manifolds, and indeed the only $1$-dimensional non-metrizable
connected manifolds \cite{Nyikos1993}, it is difficult to find examples of
dimension $>1$ (although one can do so with
$\largediamond$~\cite{Rudin1979} or $\mr{CH}$~\cite{RudinZenor}). He
therefore raised the problem of whether it was consistent that there
weren't any \cite{Nyikos1983}, \cite{Nyikos1993}. In a series of
papers~\cite{Nyikos2002,Nyikos2003,Nyikos2004a,Nyikos2004} he was
finally able to prove this from the consistency of a supercompact
cardinal, if he also assumed that the manifolds were hereditarily
collectionwise Hausdorff. We will demonstrate that neither of
these extra assumptions is  necessary:

\begin{theorem}
It is 
  consistent that every hereditarily normal manifold of dimension $>1$
  is metrizable.
\end{theorem}

 For a coherent Souslin tree $S$ (see \S 2)
PFA($S$) is the statement \cite[\S 4]{Todorcevic}:
 If $\mathcal P$ is a proper poset that preserves $S$ 
 and if $\mathcal D_\alpha (\alpha < \omega_1)$  is a sequence of
dense open subsets of $\mathcal P$ there is a filter $\mathcal
G\subset \mathcal P$ such that $\mathcal G\cap \mathcal D_\alpha \neq
\emptyset $ for all $\alpha <\omega_1$.  The notation PFA(S)[S] is
adopted in \cite{Larson} to abbreviate that we are in a forcing
extension by $S$ of a model in which $S$ was a coherent Souslin tree
and in which PFA(S) held.

\begin{theorem}
It is a consequence of PFA(S)[S] that every hereditarily normal 
manifold of dimension greater than 1 is metrizable.
\end{theorem}

We will isolate some known (quotable) consequences of PFA(S)[S].
The first, rather easy,
 is that the bounding number $\mathfrak b$ is greater than
$\omega_1$  \cite{Larson1999}. The next is the important P-ideal
dichotomy.  

\begin{defn}
  A collection $\mc{I}$ of countable subsets of a set $X$ is a
  \textrm{\bf $\mathbf{P}$-ideal} if each subset of a member of
  $\mc{I}$ is in $\mc{I}$, finite unions of members of $\mc{I}$ are in
  $\mc{I}$, and whenever $\{ I_n : n\in\w\}\subseteq\mc{I}$, there is
  a $J\in\mc{I}$ such that $I_n - J$ is finite for all
  $n$. 
\end{defn}

\begin{quotation}
  \textit{\textbf{$\mathbf{PID}$ is the statement:}
 For every
$P$-ideal  $\mc{I}$ of countable subsets of some uncountable
set $A$  either
\begin{enumerate}
\item[(i)] there is an uncountable $B\subset A$ such that
  $[B]^{\aleph_0}\subset \mc{I}$, or else
\item [(ii)] the set $A$ can be decomposed into countably many 
sets, $\{ B_n : n\in \omega\}$, such that 
 $[B_n]\cap \mc{I}=\emptyset$ for each $n\in \omega$.
\end{enumerate}
}
\end{quotation}

The consistency of $\mathbf{PID}$ does have large cardinal strength
but for P-ideals on $\omega_1$ it does not --
see  the discussion at the bottom of page 6 in
\cite{Todorcevic}. A statement similar to the $\mathbf{PID}$ for
ideals on $\omega_1$ is the one we need;
it also does not have large cardinal strength and is
weaker than the $\omega_1$ version of the statement
in  \cite[6.2]{Todorcevic}. The statement $\mathbf{P}_{22}$
was introduced in \cite{EisworthNyikos}. 
 For completeness, and to introduce the ideas we will need
for another consequence of PFA(S)[S], we include a proof
in \S 2 that  it is a consequence of PFA(S)[S].

\begin{quotation}
\textbf{$\mathbf{P_{22}}$ is the statement:\/}
  Suppose $\mc{I}$ is a $P$-ideal on a stationary subset $B$ of
  $\w_1$. Then either
\begin{enumerate}
\item[(i)]there is a stationary $E\subseteq B$ such that every
  countable subset of $E$ is in $\mc{I}$,
\item[or (ii)]there is a stationary $D\subseteq B$ such that
 $[D]^{\aleph_0}\cap \mc{I}$ is empty.
\end{enumerate}
\end{quotation}

A space $X$ is said to be $\aleph_1$-collectionwise Hausdorff 
if the points of any closed discrete subset of cardinality at most
$\aleph_1$ can be surrounded by pairwise disjoint open sets
(separated). If a separable
 space is hereditarily $\aleph_1$-collectionwise
Hausdorff,  then it can have no uncountable discrete subsets (known as
having countable spread).

The next consequence of PFA(S)[S] is:

\begin{quotation}
  \textit{\textbf{$\mathbf{CW}$:} Normal, first countable spaces are
    $\aleph_1$-collectionwise Hausdorff.}
\end{quotation}

$\mathbf{CW}$ 
was first shown to be consistent in~\cite{Tall1977};
it   was  derived
from $V=L$ in~\cite{Fleissner1974}, 
and was shown to be a consequence of 
 PFA(S)[S]  
 in~\cite{Larson}.  In fact, it is shown in \cite{Larson}
 that simply forcing with any Souslin tree will 
 produce a model of $\mathbf{CW}$. 
Let us note now that $\mathbf{CW}$ implies that any
hereditarily normal manifold   is hereditarily $\aleph_1$-collectionwise
Hausdorff. Therefore $\mathbf{CW}$ implies
that each separable hereditarily normal manifold has countable spread.

Our next axiom is our crucial new additional consequence of PFA(S)[S]:

\begin{quotation} \textit{$\mathbf{PPI}^+$: every sequentially compact
non-compact  regular space contains an uncountable free sequence.
Additionally, if the space is
 first countable, then it  contains a copy of the ordinal space
 $\omega_1$.}
\end{quotation}

Let $\mathbf{GA}$ denote the group (or conjunction)
of hypotheses:
 $\mathfrak b>\omega_1$, $\mathbf{CW}$,
  $\mathbf{PPI}^+$ and $\mathbf{P}_{22}$. We 
have, or show,
that each is a consequence of PFA(S)[S], and also establish
the desired theorem. We show in \S \ref{nolarge} that $\mathbf{GA}$ is
consistent (not requiring any large cardinals).
 
\begin{theorem} $\mathbf{GA}$ implies that all hereditarily normal
 manifolds of
  dimension greater than one are metrizable.
\end{theorem}

We acknowledge some other historical connections.

The statement $\mathbf{PPI}^+$ is a strengthening of

\begin{quotation}
  \textit{$\mathbf{PPI}$: Every first countable perfect pre-image of
    $\w_1$ includes a copy of $\w_1$.}
\end{quotation}

$\mathbf{PPI}$ was proved from $\mathrm{PFA}$ by
Fremlin~\cite{Fremlin1988}, see also e.g.~\cite{Dow1992}. Another
consequence of PFA(S)[S] relevant to this proof is

\medskip

\begin{quotation}
  \textit{$\bsigma$: In a compact $T_2$, countably tight space,
    locally countable subspaces of size $\aleph_1$ are
    $\sigma$-discrete.}
\end{quotation}

$\bsigma$ was proved from $\mr{MA}+\mathnot\mr{CH}$ by
Balogh~\cite{Balogh1983}, extending work
of~\cite{Szentmiklossy1978}. 
 $\bsigma$ implies $\mathfrak  b > \omega_1$; 
 this follows from the   result in \cite[2.4]{ppit} 
 where it is shown that  $\mathfrak 
 b = \aleph_1$  implies there is a 
 compact hereditarily separable space which is not Lindelof,
             since $\bsigma$  implies there is no such space.  
$\bsigma$ was shown to be a consequence of PFA(S)[S] in 
  \cite{FTT}.


We will need the following consequence of $\mathbf{GA}$
which is a weaker statement
than $\bsigma$. The key fact that PFA(S)[S] implies
compact, separable, hereditarily normal spaces are
hereditarily Lindel\"of  was first proven
in \cite[10.6]{Todorcevic}.

\begin{lem}
$\mathbf{GA}$ implies\label{metric} that if $X$ is a hereditarily
normal  manifold\label{noTypeII} 
then separable subsets of $X$ are Lindelof and metrizable.
\end{lem}

\begin{proof}
Let $Y$ be any separable subset of $X$ and assume
that $Y$ is not Lindelof. Recursively choose points
 $y_\alpha$, together with open sets $U_\alpha$, so that
 $y_\alpha \in Y\setminus \bigcup_{\beta<\alpha} U_\beta$,
 $y_\alpha \in U_\alpha$, and $\overline{U_\alpha}$ is separable and
 compact. Define an ideal $\mc{I}$ 
of countable subset $a$ of $\omega_1$
 according to the property that $a\in \mc{I}$ providing
 $\{ y_\alpha : \alpha \in a\}\cap U_\beta$ is finite for all
 $\beta\in \omega_1$. Since  $\mathfrak b > \omega_1$
 we have that $\mc{I}$ is a P-ideal (see \cite[6.4]{Todorcevic}). To check
 this, assume that $\{ a_n : n \in \omega\}$ are pairwise disjoint
infinite members of $\mc{I}$. For each $n$, fix an enumerating
function 
 $e_n $ from $\omega$ onto $a_n$. For each $\beta\in \omega_1$,
 there is a function $f_\beta\in \omega^\omega$ so that, for each
 $n\in \omega$ and each $m > f_\beta(n)$, $y_{e_n(m)}\notin U_\beta$. 
Using $\mathfrak b>\omega_1$, there is an $f\in \omega^\omega$ such
that $f_\beta < ^* f$ for each $\beta\in \omega_1$. 
For each $n$, let $F_n = \{ e_n(m) : m<f(n)\}$.
It follows that 
 $a = \bigcup \{a_n \setminus F_n : n\in \omega\}$ meets each
 $U_\beta$ in a finite set.  Thus $a\in \mc{I}$ and mod finite
 contains $a_n$ for each $n$.

If $B$ is any subset of $\omega_1$ such that $[B]^{\aleph_0}\subset
 \mathcal I$, then $D= \{ y_\beta  : \beta \in B\}$ is discrete 
since $D\cap U_\beta$ is finite for each $\beta \in B$. By 
$\mathbf{P}_{22}$
we must then have that there is an uncountable $B\subset \omega_1$
satisfying 
that $[B]^{\aleph_0}\cap \mc{I}$ is empty. Now let $A$ be the closure
(in $X$) of $\{ y_\beta : \beta\in B\}$. We check that $A$ is
sequentially compact.   Let $\{ x_n  : n\in \omega\}$ be any infinite
subset of $A$, we show that there is a limit point in $A$. 
Since $X$ is first countable this shows that $A$ is sequentially
compact. 
If $\{x_n : n\in \omega\}\cap \{y_\beta : \beta\in B\}$ is infinite,
 let $b\in [B]^{\aleph_0}$ be chosen so that $\{x_n : n\in \omega\}
\supset \{ y_\alpha : \alpha \in b\}$. Since $b\notin \mc{I}$, there is
a $\beta\in \omega_1$ such that $\{y_\alpha : \alpha \in b\}\cap
U_\beta$ is infinite, and so has a limit point in the compact set
 $\overline{U_\beta}$. Otherwise, we may suppose that, for each $n$,
there is an infinite $a_n\subset B$ such that $\{y_\alpha : \alpha \in
a_n\}$ converges to $x_n$. Again, using that $\mathfrak b>\omega_1$, 
similar to the verification that $\mc{I}$ is a P-ideal,
 there must be a $\beta\in \omega_1$ such that $U_\beta \cap 
 \{y_\alpha : \alpha \in a_n\}$ is infinite for infinitely many $n$. 
For any such $\beta$, there are infinitely many $n$ with 
 $x_n\in \overline{U_\beta}$. It again follows that
 $\overline{U_\beta}$ contains a limit of the sequence
 $\{ x_n : n\in \omega\}$. 
To finish the proof, we apply $\mathbf{PPI}^+$ to conclude
that either $A$ is compact or it contains a copy of
 $\omega_1$. Since $\omega_1$ contains uncountable discrete sets and
 $Y$ is separable, we must have that $A$ is compact. However, the
 final contradiction is that $A$ is not hereditarily Lindelof and so
 it cannot be covered by finitely many Euclidean open subsets of $X$.
\end{proof}

The literature on non-metrizable manifolds has identified two main types
of non-Lindel\"of manifolds, literally called Type I and Type II. A manifold is
Type II  if it is separable and non-Lindel\"of.  Lemma \ref{noTypeII} shows
that there are no hereditarily normal Type II manifolds if $\mathbf{GA}$ holds.
A manifold is said to be Type I, e.g. the Long Line,
 if it can be written as an increasing
$\omega_1$-chain, 
$\{ Y_\alpha : \alpha\in \omega_1\}$, where each $Y_\alpha$ is 
Lindel\"of, open, and 
contains the closure of each $Y_\beta$ with $\beta<\alpha$. 
In this next definition, we use the set-theoretic notion of countable
elementary submodels to help make a more strategic choice of
a representation of our Type I manifolds.  For a cardinal $\theta$,
the notation $H(\theta)$
denotes the standard set-theoretic notion of the set of all sets
that are hereditarily of cardinality less than $\theta$. These
are commonly used as stand-ins for the entire set-theoretic
universe
to avoid issues with G\"odel's famous incompleteness
theorems in arguments and constructions using elementary submodels. 
We refer the reader to any advanced book on set-theory for information
about the properties of $H(\theta)$.

\begin{defn}
 Suppose that $X$ is a non-metrizable
 manifold with dimension $n$. Let $\mathcal B_X$ denote the collection
 of compact   subsets of  $X$ that are homeomorphic to the closed 
Euclidean
 $n$-ball $\mathbb B^n$. A family $\{ M_\alpha : \alpha\in \omega_1\}$
is an elementary chain for $X$ if there is a regular
cardinal $\theta$ with $\mathcal B_X\in H(\theta)$ so that for each
 $\alpha\in \omega_1$, $M_\alpha$ is a countable elementary submodel
 of $H(\theta)$ such that 
$\mathcal B$ and each $M_\beta$ (
 $\beta<\alpha$) are members of $M_\alpha$. The chain is said to be a
continuous chain if for each limit $\alpha\in \omega_1$,
 $M_\alpha = \bigcup_{\beta < \alpha } M_\beta$. 

Whenever $\{ M_\alpha
 : \alpha \in \omega_1\}$ is an elementary chain for $X$, 
let $X(M_\alpha)$ denote the union of the collection $\mathcal B_X
\cap M_\alpha$.  
\end{defn}

Here is the main reason for our preference to use elementary submodels
in this  proof. Again the main ideas are from \cite{Nyikos2004}, but the
proof using elementary submodels
is much simpler.  Throughout the paper the term 
\textit{
component\/} refers to
the
standard notion of  connected component.

\begin{lem} 
Suppose that $X$ is a non-metrizable hereditarily normal
 manifold of dimension $n>1$. Let $\theta$ be
a large enough regular
cardinal $\theta$ so that\label{elementary}
 $\mathcal B_X\in H(\theta)$
and let $M$ be  a countable elementary submodel
 of $H(\theta)$ such that 
$\mathcal B_X$ is a  member of $M$.
Then $X(M) = \bigcup (M\cap \mathcal B_X)$ is 
an open Lindelof subset of $X$ with the property 
that every
 component of the non-empty boundary,
 $\partial X(M)$, is non-trivial.
\end{lem}

\begin{proof}
  We let $\mathcal{B}_X$ denote the family of all homeomorphic copies
  of the closed unit ball of~$\mathbb{R}^n$ in~$X$.
  As $X$~is a manifold this family is such that whenever $O$ is open in~$X$
  and $x\in O$ there is a $B\in\mathcal{B}_X$ such that $x$ is in the interior
  of~$B$ and $B\subseteq O$.
  
 Let $Y$ denote the set $X(M)$. Since $Y$ is metrizable, and
 $X$ is not, $Y$ is a proper subset of $X$.
 Each member of $\mathcal B_X\cap M$ is separable and hence
 $B\cap M$ is dense in~$B$ whenever $B\in\mathcal B_X\cap M$;
 it follows that $Y\cap M$ is a dense subset of~$Y$. 

 We also note that $Y$ is open since if $B\in \mathcal B_X\cap M$,
 then $B$ is compact and so is contained in the interior of a finite union
 of members of $\mathcal B_X$. By elementarity,
 there is a such a finite set in $\mathcal B_X\cap M$. 
Similarly, if $\mathcal B'$ is a finite subset of $\mathcal B_X\cap M$, 
then, by elementarity, each Lindel\"of  component of 
 $X\setminus \mathcal B'$ that meets $Y\cap M$ will be
a subset of~$Y$. More precisely, if $C$ is such a component
and if $y\in C\cap M$, then $M$ will witness
that there is a countable collection of members 
of $\mathcal B_X$ that covers the component of $y$ in $X\setminus
 \mathcal B'$.
Also, we have that $X$ itself must have non-Lindel\"of 
 components since   Lindel\"of
subsets of any manifold  are  metrizable while $X$, being locally connected, is
the free union of its  components.  

Now we choose any $x$ in $\partial X(M) = \partial Y = 
\overline{Y} \setminus  Y$.
Take any $B\in \mathcal B_X$ with $x$ in its interior. 
We  assume, working towards a contradiction,
that the component of $x$ in $\partial Y$ is $\{x\}$.

Since $\partial Y\cap B$ is compact and $\{x\}$ is a 
component of $\partial Y\cap B$, we can split the latter set into two
relatively clopen sets $C$ and $D$, where $C$~is the union of all components
of $\partial Y\cap B$ that meet the boundary of~$B$ and $D$~is its complement.
For now we allow for the possibility that $C=\emptyset$ but $D$~is not empty
as it contains~$x$.
We choose $W$ with $D\subseteq W$ and such that $\overline {W}$
is contained in the interior of~$B$ and disjoint from~$C$.

Note that $W$ and $\overline{W}$ are Lindel\"of because $B$ is compact and
hereditarily Lindel\"of, being homeomorphic to the unit ball of~$\mathbb{R}^n$.
Since $\partial W$ and $\partial Y$ are disjoint the set
$\partial W\cap Y$ is closed and hence compact.
There is a finite subfamily~$\mathcal B_1$ of $M\cap \mathcal B_X$ whose
union contains $\partial W\cap Y$. 
The complement $W\setminus \bigcup\mathcal B_1$ is a neighbourhood of~$x$,
so it meets $Y\cap M$.
The component, $E$, of $x$ in this complement is Lindel\"of but not contained
in $Y$, therefore $E$ is not a component of $X\setminus \bigcup\mathcal B_1$
which implies that $E\setminus \overline{Y}$ is not empty.
Since $\dim E>1$, $x$ can not be a cut-point of $E$ and so it 
 follows that  $x$ is not the only point of $E\cap \partial Y
\subset W\cap \partial Y$.

This means that we can choose  disjoint open subsets of $W$,
say $O_1$ and $ O_2$, each also having compact non-empty intersection with
$\partial Y$ and whose boundaries miss~$\partial Y$.
Fix points $z_1\in O_1\cap \partial Y$ and $z_2\in O_2\cap \partial Y$.

Now $Y\cap(\partial O_1 \cup \partial O_2 )$ is compact and again can be
covered by some $K$ where $K$ is equal to a union of some finite
subfamily~$\mathcal B_2$ of~$M\cap \mathcal B_X$.
Also since $Y\cap (\partial O_1\cup \partial O_2)$ is disjoint
from the boundary of $B$, we can ensure that $K$ is disjoint
from the boundary of $B$. 
Since $K$ is a compact subset of $Y$, each  component of
$O_1\setminus K$ and $O_2\setminus K$  meets $Y\cap M$;
so choose points $y_1$ and $y_2$ in $Y\cap M$ that are in the components
in $O_1\setminus K$ and $O_2\setminus K$ of $z_1$ and $z_2$ respectively.
Let $C_1$ and $C_2$ be the  components in $X\setminus K$ of $ y_1$
and $ y_2$ respectively.
Neither component is contained in $Y$ and so neither is Lindel\"of.
Thus, neither is contained in $B$ and so they both meet the (arcwise)
connected boundary of~$B$.
Since  components of $B\setminus K$ are path-connected, there is a
path in $X\setminus K$ from $y_1$ to~$y_2 $.
By elementarity there is such a path in~$M$ and such a path would lie
completely within~$Y$
(the path is covered by a finite subfamily of~$\mathcal{B}_X$ and one such
family should be in~$M$).
However, $Y\cap (O_1\setminus K)$ is clopen in $Y$ so there is no path
in~$Y$ that connects~$y_1$ and~$y_2$.
This contradiction finishes the proof. 
\end{proof}

This next corollary is the representation as a Type I sub-manifold that we
require.
                                        
\begin{corollary}
Suppose that $X$ is a non-metrizable hereditarily normal manifold of
dimension greater than 1. Then there is an increasing chain
 $\{ Y_\alpha : \alpha \in \omega_1\}$  of open Lindelof subsets
satisfying that\label{component}
\begin{enumerate}
\item for each $\alpha$, the boundary $\partial Y_\alpha$ is non-empty
  and contained   in $Y_{\alpha+1}$, 
\item for each $\alpha$, each  component of $\partial
  Y_\alpha$ is non-trivial,
\item for limit $\alpha$, $Y_\alpha = \bigcup \{Y_\beta : \beta\in
  \alpha\}$.
\end{enumerate}
Additionally, the union $ \bigcup\{ Y_\alpha : \alpha\in
\omega_1\}$ is closed (and open) in $X$.
\end{corollary}

\begin{proof}
Fix a continuous elementary chain $\{ M_\alpha : \alpha\in
\omega_1\}$ for $X$. Fix any $\alpha\in \omega_1$.
By Lemma \ref{elementary}, $Y_\alpha = X(M_\alpha)$ is Lindelof with non-empty
boundary, $\partial X(M_\alpha)$, and 
each
 component in  $\partial X(M_\alpha)$
is non-trivial.
By Lemma \ref{metric}, $\overline{X(M_\alpha)}$
is Lindelof, and so by elementarity,
  $M_{\alpha+1}\cap \mathcal B_X$ is a cover  
of $\overline{X(M_\alpha)}$. Finally, $\bigcup \{ Y_\alpha : \alpha
\in \omega_1\} $
 is closed 
 because any $x\in X$ that is in the closure will be in
 $\overline{Y_\alpha}\subset Y_{\alpha+1}$ for some $\alpha\in
 \omega_1$. 
\end{proof}

Now we are ready to give a proof of the main theorem. The 
clever topological  ideas of the 
proof are taken from  \cite[p189]{Nyikos2002}.
 A sketch of this proof appears in~\cite{Tall2005}. The main idea of
 the proof is to use $\mathbf{PPI}^+$ to find  copies of
 $\omega_1$ and, combined with Lemma \ref{component}, to show 
 that, in fact, there
 are copies of the  Tychonoff plank in the space. It is easily
shown that the Tychonoff plank is not hereditarily normal.

\begin{theorem} The statement $\mathbf{GA}$ implies that
each   hereditarily normal  manifold of dimension
  greater than 1 is metrizable.
\end{theorem}

\begin{proof}
Assume that $X$ is a non-metrizable hereditarily normal
 manifold of dimension greater
than 1. Let $\{ Y_\alpha : \alpha \in \omega_1\}$ be chosen
 as in Corollary \ref{component}.
For each $\alpha\in \omega_1$, choose any
 point $x_\alpha\in \partial Y_\alpha$. 
  It is immediate that $\overline{\{ x_\alpha :
   \alpha\in \omega_1\}}$ is nowhere dense in $X$.
Also let
 $\{ U_\alpha : \alpha \in \omega_1\}\subset \mathcal B_X$ be
any selection so that $U_\alpha\subset Y_{\alpha+1}$
and $x_\alpha$ is in the interior of $U_\alpha$. 
 We first show that
if $E\subset  \omega_1$ is stationary, then 
$D= \{ x_\alpha : \alpha \in E\}$ is not  discrete. 
 For each limit $\alpha$, using item (3) of Corollary \ref{component},
there is a $\beta_\alpha < \alpha$ such that $U_\alpha\cap
Y_{\beta_\alpha}\setminus \overline{D}$
 is not empty. By the pressing down lemma, 
 there is a fixed $\beta$ such that $\beta = \beta_\alpha$ for
 uncountably many $\alpha\in E$. Since $Y_\beta\setminus \overline{D}$
 is separable,
 there are $\alpha, \alpha'\in E$ such that $U_{\alpha}\cap
 U_{\alpha'}\cap Y_\beta$ is not empty. 
The choice
 of the sequence of $U_\alpha$'s was (basically) arbitrary, 
and so it follows that  $D$ can not be separated 
Since $D\cup (X\setminus \overline{D})$ is first countable,
and thus $\aleph_1$-collectionwise Hausdorff,
 this shows that $D$ cannot be discrete. 

Define the ideal $\mc{I}$ by $a\in \mc{I}$ if $a\in
[\omega_1]^{\aleph_0}$ and, for all $\beta\in \omega_1$, 
 $\{ x_\alpha : \alpha \in a\}\cap U_\beta$ is finite. As before,
 $\mc{I}$ is a P-ideal on $\omega_1$. If $A\subset \omega_1$ satisfies
 that 
 $[A]^{\aleph_0}\subset \mc{I}$, then $D = \{ x_\alpha : \alpha \in
 A\}$
is discrete. Therefore there is no such stationary $A$, and
so by  $\mathbf{P}_{22}$, there is a
stationary subset $A$ of $\omega_1$ such that 
 $ [A]^{\aleph_0}\cap \mc{I}$ is empty. It again follows that
 $X_A = 
\overline{\{x_\alpha : \alpha \in A\}}$ is sequentially compact. Let
 us choose, by  
applying $\mathbf{PPI}^+$, a  copy $W$
 of $\omega_1$ contained in $X_A$. Let 
$W = \{ w_\xi : \xi\in\omega_1\}$ be the homeomorphic indexing of
$W$.  
For each $\alpha\in \omega_1$, Lemma \ref{metric} implies that
 $\overline{Y_\alpha}$ is Lindelof and, by elementarity,
 contained in $Y_{\alpha+1}$. 
Therefore,  we have that,  for each $\alpha$,
$W\cap Y_\alpha$ is countable, and its closure is contained in 
 $Y_{\alpha+1}$.  It follows that there is a cub
 $C\subset\omega_1$ satisfying that for each
 $\gamma<\delta$ both in $C$, the set $\{ w_\beta : \gamma\leq
 \beta<\delta\}$ is contained in   $Y_\delta\setminus
 Y_\gamma$.  Therefore 
 $\{ w_\gamma : \gamma\in C\}$ is another copy of $\omega_1$ with
the property that $w_\gamma \in \partial Y_\gamma$ for each 
 $\gamma\in C$. 

For each $\gamma\in C$, apply Lemma \ref{component}, so as to choose
infinite compact connected $K_\gamma\subset \partial Y_\gamma$ with
 $w_\gamma\in K_\gamma$. Make another selection $y_\gamma\in
 K_\gamma\setminus W $ arbitrarily. Now choose, 
for each $\gamma\in C$, 
a basic 
 set  $V_\gamma\in \mathcal B_X$  so that $y_\gamma$ is in the interior of
 $V_\gamma$ and $V_\gamma \subset X\setminus W$. Proceeding as we did
 with the sequence of $\{ x_\alpha : \alpha\in \omega_1\}$, there is a
 stationary set $A_1\subset C$ so that 
$\{ y_\alpha : \alpha \in A_1\}$ has sequentially compact closure. 
Since $\overline{Y_{\gamma}}$ is Lindelof and contains
 $\{ y_\alpha : \alpha \in A_1\cap \gamma\}$ for each $\gamma\in C$,
 it follows then that the closure of 
 $\{ y_\alpha : \alpha \in A_1\cap \gamma\}$ is compact and disjoint
 from $W$ for each $\gamma\in C$. Since $X$ is first countable, this
 also implies that the closure of the entire set
$\{ y_\alpha : \alpha \in A_1\}$  is disjoint from the closed set $W$.  
Since $X$ is normal, there is a continuous function $f$ from $X$ into
$[0,1]$
such that $f[W]=\{1\}$ and
$f(y_\alpha) = 0$ for all $\alpha\in A_1$. Note that $f[K_\alpha] =
[0,1]$
for each $\alpha\in A_1$. Finally, using $f$ we will show there is a
non-normal subspace for our contradiction. For each $\alpha\in A_1$, 
 choose, yet another, point $z_\alpha \in K_\alpha$ in such a way
that the map $f$ restricted to $\{ z_\alpha : \alpha\in A_1\}$ is 
one-to-one. Repeating the steps above, there is a stationary set
 $A_2\subset A_1$ so that the closure of each countable subset
 of $\{ z_\alpha : \alpha \in A_2\}$ is compact. Let $Z$ denote
the closure of the set $\{ z_\alpha : \alpha \in A_2\}$,
and for each $r\in [0,1]$, let 
 $Z_r = f^{-1}(r)\cap Z$.  We will use the following property of these
 subsets of $Z$.
Consider any open set $U$ of
 $X$ that contains $Z_r\cap \partial Y_\gamma$ for any $r\in [0,1]$
and $\gamma\in C_\omega$. Since $Z_r \cap Y_\gamma$ has compact
closure, there is a $\beta <\gamma$ such that
 $Z_r\setminus Y_\beta$ is contained in $U$. By the pressing down
 lemma, given any open $U$ containing $Z_r\cap Y_\gamma$ for all
any stationary set of $\gamma\in \omega_1$, there is a $\beta\in
\omega_1$ such that $Z_r\setminus Y_\beta$ is contained in $U$.

 Choose any  $r\in  [0,1]$
such that $r$ is a complete accumulation point of $\{ f(z_\alpha) :
\alpha \in A_2\}$. 
Choose any sequence
 $\{r_n : n\in \omega\}$ converging to $r$ so that each $r_n$
is also a complete accumulation point of $\{ f(z_\alpha) : \alpha \in
A_2\}$. There is a common cub $C_\omega$ such
that $Z_{r_n}\cap \partial Y_\gamma$ 
and $Z_r\cap \partial Y_\gamma$ is not empty for each 
 $n\in \omega$ and $\gamma \in C_\omega$. 
Let $Z_r(C_\omega') = \{ Z_r\cap \partial Y_\gamma : \gamma \in
C_\omega'\}$ 
where $C_\omega'$ is the set of relative limit points of
$C_\omega$. Since 
$Z_r$ is closed in $Z$, it follows that
 $Z_r\setminus Z_r(C_\omega')$ is  a closed subset of $Z\setminus
 Z_r(C_\omega')$. We also note that $H = Z \cap
 \bigcup \{ \partial Y_\gamma : \gamma\in C_\omega'\}$ is a closed
subset of $Z$, and so $H\setminus Z_r(C_\omega')$ is a  closed
subset of $Z\setminus Z_r(C_\omega')$. We show that $Z_r\setminus 
 Z_r(C_\omega')$ and $H\setminus Z_r(C_\omega')$ can not be separated by 
disjoint open subsets of $Z\setminus Z_r(C_\omega')$. Since 
$Z_r\setminus Z_r(C_\omega')$ and $H\setminus Z_r(C_\omega')$ are disjoint, 
this will complete the proof. 
Suppose that $U$ is an open subset of $Z\setminus Z_r(C_\omega')$ 
that contains $H\setminus Z_r(C_\omega')$.  By the above mentioned
property of each $Z_{r_n}$, we have that there is a $\beta\in
\omega_1$
such that $Z_{r_n}\setminus Y_\beta$ is contained in $U$ for each
$n\in \omega$. Choose any $\beta < \gamma\in C_\omega \setminus
C_\gamma'$. For each $n$, choose $z_n \in Z_{r_n}\cap \partial
Y_\gamma$. Since $Z\cap \partial Y_\gamma$ is compact, let
 $z$ be any limit point of $\{ z_n : n\in \omega\}$. By the continuity
 of $f$, $f(z) = r$ and so $z\in Z_r\cap \partial Y_\gamma$. In
 other words, $z\in Z_r\setminus Z_r(C_\omega')$, completing
the proof that $H\setminus Z_r(C_\omega')$ and $Z_r\setminus
Z_r(C_\omega')$ can not be separated by open sets.
\end{proof}

\section{on $\mathbf{P}_{22}$}

As usual $S$ is a coherent Souslin tree. For us, it will
be a full branching downward closed subtree of $\omega^{<\omega_1}$. 
Naturally it is a Souslin tree (no uncountable antichains) and 
has the additional property

\begin{quotation}
 for each $s\in S$ and $t\in \omega^{<\omega_1}$ with $\dom(t) =
  \dom(s)$, $t$ is in $S$ if and only if
 $\{ \xi \in \dom(s) : s(\xi)\neq t(\xi)\}$ is finite.
\end{quotation}

In a forcing argument using $S$ as the forcing poset, we
 will still use $s < s'$  to mean that $s \subset s'$, and so, 
 $s'$ is a stronger condition.  We will also use
the more compact notation   $o(s)$ to
 denote the order-type of $\dom(s)$ for $s\in S$.
Now we give a proof that our statement $\mathbf{P}_{22}$ is a
consequence of PFA(S)[S] following \cite[6.1]{Todorcevic}.

Here is a simple standard fact about forcing with a Souslin tree
 that we will need repeatedly.

\begin{lemma} Suppose that
$S$ is a Souslin tree and $S\in M$ for some countable
  elementary\label{basiclemma}
 submodel of any $H(\theta)$ ($\theta \geq {\omega_2}$).
If $\dot x, \dot X \in M$ are Souslin names, and $s\in S\setminus M$,
  then there is an $s'< s$ with $s'\in M$ such that
\begin{enumerate}
\item $s\Vdash \dot X = \emptyset$ if and only if  $s' \Vdash \dot X
  =\emptyset$,
\item $s\Vdash \dot x \in \dot X$ if and only if $s'\Vdash \dot x\in
  \dot X$.
\end{enumerate}
\end{lemma}

\begin{proof}
The second item follows from the first (by simply considering
the set $\dot X \cap \{\dot x\}$) so we consider any $\dot X$ in $M$.
Since $S$ is a ccc forcing and the set of conditions that decide
the statement ``$\dot X = \emptyset$'' is dense and open, there is a
$\gamma\in M\cap\omega_1$ such that each element of $S_\gamma$ decides  
this statement. Therefore $s\restriction \gamma$ decides the statement
and, since $s$ is a stronger condition than $s\restriction \gamma$, 
they assign the same truth value to the statement.
\end{proof}

Note, for example, Lemma \ref{basiclemma} can be used to show that
 if $\dot E\in M$  is an $S$-name of a subset of
 $\omega_1$
 and $s\Vdash M\cap \omega_1\in \dot E$,
  then $s\Vdash \dot E $ is stationary.
To see this   we can let $\dot X$ denote the set of 
(ground model) 
 cub subsets of $\omega_1$ that are disjoint from $\dot E$.
 Then, if $s\Vdash M\cap \omega_1 \in \dot E$, we have
 that   for all
    cub $C$ in $M$, $s\Vdash C\cap \dot E$ is not empty.
So,   if $s'<s$ is in $M$, we have that $s'$ forces
 that $\dot X$ is empty, and $\dot E$ is stationary.

\begin{proposition}  Assume PFA(S) then $S$ forces
 that $\mathbf{P}_{22}$ holds.
\end{proposition}

\begin{proof}
Let $\dot {\mathcal I}$ be the $S$-name for a P-ideal on a stationary
subset $B$ of  $\omega_1$
and assume that some $s_0\in S$ forces that $\dot {\mathcal I}\cap 
 [E]^{\aleph_0}\neq \emptyset$ for all stationary sets $E$. 
If $s_0$ also forces that $\dot {\mathcal I}$ is  a counterexample to 
 $\mathbf{P}_{22}$, then using that $S$ is homogeneous
 and the forcing maximum principle, we can assume
 that $s_0$ is the root of $S$ and just show that $\dot {\mathcal I}$
 is not a counterexample.
Fix any well-ordering 
 $\prec$ of $H(\aleph_2)$.

\begin{claim}
For each countable elementary submodel 
 $M$ of $(H(\aleph_2),\prec)$  and each $s \in S_{M\cap \omega_1}$,
 there is a set $a(s,M)$ such that 
$s\Vdash a(s,M)\in \dot {\mathcal
   I}$ and $s\Vdash a \subset^* a(s,M)$ for all $a\in 
M\cap  \dot{\mathcal M}$.
\end{claim}

Proof of Claim 1:  Since $s$ forces that $\dot{\mathcal I}$ is a
P-ideal, there is a $\prec$-minimal name $\dot a$ such that
 $1$ forces that each member of 
$M\cap \dot{\mathcal I}$ is a subset mod finite of $\dot a$. Since $S$
is ccc, there is a countable maximal 
antichain $\{ s_n : n\in \omega\}$ and a countable family 
 $\{ a_n : n\in \omega\}$ of countable subset of $\omega_1$ such
that, for each $n$, $s_n\Vdash \dot a =a_n $. Furthermore, $s$ forces
a value on each member of $M\cap \dot {\mathcal I}$. Let 
 $\mathcal J$ denote the countable family of sets forced by
 $s$ to be members of $M\cap \dot{\mathcal I}$. Note that every member
 of $\mathcal J$ is mod finite contained in every member of 
$\{ a_n :n\in \omega\}$. We may choose $a(s,M)$ to be the
$\prec$-minimal set that splits this $(\omega,\omega)$-gap.

\bigskip

One change from \cite{Todorcevic} is that we
 begin with a partition
 $\mathcal E = \{ E_s : s\in S\}$  of 
  $\omega_1$ by stationary sets
 so that, in addition, $ E_s\subset B$   for each 
 $s\in S$ other than the root $\emptyset$.
 Thus $\bigcup \{ E_s : s\in S \setminus \{\emptyset\}\}$ contains
  $\omega_1\setminus B$.
 We also require that $\dom(s) < \delta$ for all 
limit $\delta\in E_s$.
Then we 
let $\mathcal P$ be the collection of all mappings of the form
 $p:\mathcal M_p \rightarrow S$,  where

\begin{enumerate}
\item $\mathcal M_p $ is a finite $\in$-chain of countable elementary
  submodels of $(H(\aleph_2),\mathcal E, \prec)$
\item $M\in \mathcal M_p$ and $\delta = M\cap \omega_1\in E_s$
implies $s< p(M)\in S_{\delta}$,
\item $M\in N\in \mathcal M_p$ implies $a(p(M),M)\in N$.
\end{enumerate}

We let $p\leq q$ if, 

\begin{enumerate}
\addtocounter{enumi}{3}
\item $\mathcal M_p\supset \mathcal M_q$ and $q = p\restriction
  \mathcal M_q$,
\item $
  N\cap\omega_1\in a(q(M),M)$ whenever $N\cap\omega_1\notin E_{\emptyset}$,
  $p(N) < q(M)$
with  $M\in N\in \mathcal M_q$,
  and $M\in  \mathcal M_p\setminus
\mathcal M_q$. 
\end{enumerate}

In order to apply PFA(S) to $\mathcal P$, 
we have to show that $\mathcal P$ is a proper poset that preserves
that $S$ is Souslin. Once we do, we let $\mathcal G$ be a filter on
$\mathcal P$ that meets sufficiently many (no more than $\omega_1$)
dense subsets to
ensure that there is a cub $C\subset \omega_1$ such that for each
 $\delta\in C$, there is a $p_\delta\in \mathcal G$ and an $M_\delta\in
 \mathcal M_{p_\delta}$ with $M_\delta\cap \omega_1 = \delta$.  The
 role of the family $\mathcal E$ is to ensure the next Claim.

\begin{claim}
Each $s_0\in S$ forces that
the $S$-name $\dot E = \{ \delta\in B : p_\delta(M_\delta) \in \dot g\}$ is
a stationary subset of $\omega_1$, where $\dot g$
is the $S$-name of the generic branch through $S$. 
\end{claim}

Proof of Claim 2:   It suffices to show
 that $s_0$ does not force that $\dot E$ is not stationary
 by finding an extension that forces $\dot E$ is stationary.
Choose any $\delta \in C\cap E_{s_0}$. We have that $s = p_\delta(M_\delta)$ 
forces that $\delta = M_\delta\cap \omega_1$ is in $\dot E$.
Also, since $\delta \in E_{s_0}$, we have,
from the definition of $\mathcal P$, that $s_0 < s$.
By Lemma \ref{basiclemma}, as explained in 
the discussion immediately following it, we have that 
 $s$ forces that $\dot E$ is stationary. 
 \bigskip

\begin{claim} Each $s\in S$ forces that
$[\dot E]^{\aleph_0} \subset \dot{\mathcal I}$, where
 $\dot E$ is defined in Claim 2.
\end{claim}

Proof of Claim 3: It suffices to show that if $\gamma\in  \omega_1$
and $s\in S_\gamma$, then $s\Vdash \dot E\cap \gamma \in \dot{\mathcal
  I}$. Recall that there is a $\delta > \gamma$ such that
$s<p_\delta(M_\delta)$. By the definition of the ordering on
$\mathcal P$ (item (5)) we have that $\{ \gamma \in \dot E : 
 p_\gamma(M_\gamma) < p_\delta(M_\delta) \ \mbox{and}\
M_\gamma\notin \mathcal M_{p_\delta}\}$ is contained in 
 $a( p(M_\delta), M_\delta)$.  Therefore, $p(M_\delta)$ forces
that $\dot E\cap \delta \in \mathcal I$. 

\bigskip

We finish the proof of the Proposition by proving that $S\times \mathcal
P$ is proper. Let $M$ be any countable elementary submodel of
$H(\kappa)$ for some regular $\kappa > \omega_2$. We show that any pair
$(s^\dagger, q)$ where $s^\dagger\in S\setminus M$ and $M\cap
H(\aleph_2)\in  \mathcal M_{p_0}$ is an $M$-generic condition for
$S\times \mathcal 
P$. Consider any dense open set $D$ of $S\times \mathcal P$ that is a
member of $M$. By extending the condition $(s^\dagger,q)$ we can assume
that  $(s^\dagger,q)$ is in $D$ and that there is some countable
elementary submodel of $H(\kappa)$ containing $q$ but not $s^\dagger$.
It is useful to regard $D$ as an $S$-name $\dot D$
of a dense open subset of
$\mathcal P$ in the sense that if $(t,p)\in D$, then we interpret
this as $t\Vdash p\in \dot D$. 

It is evident from conditions (4) and (5) of the definition of 
$\mathcal P$ that $q_0 = q\restriction M\in M$ and
that $q$ is an extension of $q_0$. Let $\delta = M\cap \omega_1$.
Let $\{ M_1, \ldots , M_{\ell}\}$
be an increasing enumeration of $\mathcal M_q\setminus M$. Of course
$M_1 = M\cap H(\aleph_2)$. Let $\{ s_0, \ldots , s_m\}$ be any
one-to-one list of the set $\{ s^\dagger \restriction \delta,
 q(M_1)\restriction \delta, \ldots , q(M_{\ell})\restriction
 \delta\}$
so that $s_0 = s^\dagger\restriction \delta$. For each $1\leq j\leq \ell$,
 let $m_j $ denote the value such that $s_{m_j} = q(M_j)\restriction
 \delta$. 
Let $J$ denote those $1\leq j\leq\ell$ such that 
 $q(M_j)\restriction [\delta, M_j\cap \omega_1) \subset s^\dagger$.
To avoid trivialities, we can assume that we extended $(s^\dagger, q)$
if necessary, 
so as to have that $J$ is not empty.

Since $S$ is a coherent Souslin tree, there is a
 $\bar\delta\in M$ such that $s_0\restriction [\bar\delta, \delta) =
 s_i \restriction [\bar\delta, \delta)$ for each $i\leq m$. By
increasing $\bar\delta$ we can also ensure that $\bar M\cap \omega_1 <
\bar \delta$ for each $\bar M\in \mathcal M_{q_0}$. 
Let $\bar s_i = s_i\restriction \bar \delta$ for $i\leq m$, and
 notice that $\{ \bar s_0,\ldots, \bar s_m\}\in M_0$. 
For each $s\in S$ with $\bar\delta\leq \dom(s)$, let
 $\bar s_0 \oplus s$ denote the function
 $\bar s_0\cup s\restriction [\bar \delta, \dom(s))$; since $S$ is a
 coherent Souslin tree $\bar s\oplus s\in S$.
Note that $J = \{ j <\ell : \bar s_0\oplus p(M_j) < s^\dagger \}$. 
Also, define $J_B$ to be the set $\{ j\in J : M_j\cap \omega_1\notin
 E_\emptyset\}$.

Say that $(t,p)\in D$ is \underbar{like} $(s^\dagger,q)$ providing
\begin{enumerate}
\item there is a $M^p_0\in \mathcal M_p$ such that $\bar\delta \in
  M^p_0$ and  $q_0  = p \restriction M^p_0$,
\item  $\mathcal M_p\setminus M^p_0 $ has size $\ell$, enumerated as
 $\{ M^p_0,\ldots M^p_{\ell-1}\}$ in increasing order,
\item $\bar s_{i_j}< p(M^p_j)$ for $j<\ell$
\item $J = \{ j < \ell : \bar s_0 \oplus p(M^p_j) < t \}$,
\item $J_B = \{ j\in J : M^p_j\cap \omega_1 \notin E_{\emptyset}\}$.
\end{enumerate}

Our proof that
$S\times \mathcal P$ is proper will depend on finding some
$(t,p)\in D\cap M$ that is \underbar{like} $(s^\dagger,q)$
and, in addition, is compatible with $(t_0,q)$. Of course this
requires that $t< s^\dagger $, 
but what else? Since $\mathcal M_p\in M_0$
and $p<q_0$ we automatically have that $\mathcal M_p\cup \mathcal M_q$
is an $\in $-chain. The most difficult (and remaining)
 requirement is to ensure
that if $p(M^p_j) < q(M_k)$ then $M^p_j\cap \omega_1 $
must be in $
a(q(M_k),M_k)$ if $M^p_j\notin E_{\emptyset}$.
 Interestingly, the values of $1\leq j\leq \ell$ that
we will have to worry about are exactly those values in $J_B$
(in most proofs it would be all values of $J$). This is
because we must have that $p(M^p_j) < s_{i_k}$ and so $\bar s_0 \oplus
 p(M^p_j) < s_0 < s^\dagger$. Since also, $t<s^\dagger$ and $t\in M_0$, we
have that $\bar s_0\oplus p(M^p_j) < t$, which is the requirement
that $j\in J$. One frequently troublesome aspect to these proofs is
that the values of $k$ for which $p(M^p_j) < q(M_k)$ will be
all  $k$ such that $i_k = i_j$, not just values of $k$ in $J$.
For easier reference in the remaining proof, let
 $a_k = a(q(M_k) , M_k)$ for $1\leq k \leq m$.

The set $\mathbf{L}\subset D$ consisting of those pairs $(t,p)$ that
are 
\underbar{like} $(s^\dagger,q)$ is an element of $M$. 
For each
 $(t,p)\in \mathbf{L}$, let $T_{t,p} = \langle t_0, t_1, \ldots, t_\ell\rangle$
be a re-naming of $\langle t, p(M^p_1), p(M^p_2),\ldots,
p(M^p_\ell)\rangle$.  Let $\mathcal T(\mathbf {L})$ denote the 
set $\{ T_{t,p} : (t,p)\in \mathbf{L}\}$, and for each $1\leq j\leq
m$,    let $\mathcal T(\mathbf {L})_j = 
\{  \vec t \restriction j : \vec t
= \langle t_0, t_1, \ldots, t_\ell\rangle \in \mathcal
T(\mathbf{L})\}$.
Of course $\mathcal T(\mathbf{L})_\ell $ is equal to $\mathcal
T(\mathbf{L})$.  Since $D$ is an open subset of
 $S\times\mathcal P$, let us note that if $
 \langle t_0, t_1, \ldots, t_{j-1}\rangle \in \mathcal
T(\mathbf{L})_j$,
then $ \langle \bar t_0, t_1, \ldots, t_{j-1}\rangle \in \mathcal
T(\mathbf{L})_j$ for all $\bar t_0 > t_0$.

Now we want to  use $\mathcal T(\mathbf {L})$ to define an
$S$-name of  a subset of $[\omega_1]^{\leq\ell}$. 
For $\vec t = \langle t_0, t_1,\ldots, t_{j-1}\rangle
\in \mathcal T(\mathbf{L})_j$  ($1\leq j\leq \ell$),
 let $\Delta_{\vec t}$ be the sequence $ \langle \delta_1, \ldots, 
 \delta_{j-1}\rangle$ where $\delta_i = \dom(t_i)$.
We define $\dot F_\ell$ to be the $S$-name consisting of all
pairs $(t_0, \langle \delta_1, \ldots , \delta_\ell \rangle)$
for which there is a $\vec t = \langle t_0, \ldots, t_\ell\rangle$
in $\mathcal T(\mathbf{L})$ such that $\Delta_{\vec t} = 
\langle \delta_1, \ldots , \delta_\ell\rangle$. In saying
that $\dot F_\ell$ is an $S$-name we are adopting
 the standard abuse of notation that 
an element of the ground model can be used as an $S$-name for itself.
 By reverse
induction on $\ell > k \geq 1$, we define $\dot F_k$. 
Having defined $\dot F_{k+1}$, we define $\dot F_k$. 
If $k+1\notin J_B$, then $\dot F_k = \dot F_{k+1}$. 
If $j = k+1$ is in $J_B$, then 
$(t_0, \langle \delta_1, \ldots, \delta_{k} \rangle)$ 
is in $\dot F_k$ providing $t_0$ forces that the set
$$\dot F_j( \langle \delta_1, \ldots, \delta_{k} \rangle) = 
\{ \gamma : (\exists \bar t_0)(\exists \vec{\delta}~)~~
  (t_0,\vec \delta~) \in \dot F_j \ \mbox{and}\ 
\langle \delta_1, \ldots, \delta_k, \gamma\rangle =
\vec\delta\restriction j  \}$$ is stationary.

The next, somewhat standard, step is to prove that,
 for each $k<\ell$ with $k+1\in J$, 
$s^\dagger \Vdash \Delta_{T_{s^\dagger,q}} \restriction k \in 
 \dot F_k$.
Again, this is by reverse induction on $\ell>k\geq 0$.  
Let $\vec\gamma = \Delta_{T_{s^\dagger, q}} = \langle \gamma_1,\ldots , 
\gamma_\ell\rangle$. Certainly, $s^\dagger \Vdash
 \vec\gamma \in \dot F_\ell$.  We again take note
of the fact that $\dot F_k\in M_0$ for each $0\leq k\leq \ell$. 
Let $J_B=\{j_1, \ldots , j_{\bar \ell}\}$
 be listed in increasing order. For $j_{\bar \ell}\leq k\leq \ell$,
 we have that $s^\dagger\Vdash \dot F_k = \dot F_\ell$. 
Now let $j = k+1 = j_{\bar \ell}$ and observe that
 $\dot F_{j}(\vec \gamma\restriction k)$ is a member of the model 
 $M_{j}$, and that $\gamma_{j} = M_{j}\cap \omega_1 $
is forced by $s^\dagger$ to be an element of
 $\dot F_{j}(\vec \gamma\restriction k)$.  We show that
this means that $s^\dagger$
forces that 
 $\dot F_{j}(\vec \gamma\restriction k)$ is stationary.
Within $M_j$, there is a maximal antichain (in fact a level) of $S$
 with the property that each member of the antichain decides
whether or not 
 $\dot F_{j}(\vec \gamma\restriction k)$ is stationary. For each such
 node that decides that it is not stationary, 
there is a cub in $M_j$ that is forced to be disjoint. 
Since $\gamma_j$ is in every cub from $M_j$ and since
 $s^\dagger$ forces that $\gamma_j$ is in $\dot F_j(\vec
 \gamma\restriction k)$, we have that it is forced to be stationary. 
This completes the inductive step that $s^\dagger$ forces
that $\vec\gamma\restriction k$ is in $\dot F_k$.

To complete the proof, we work our way back up 
from $\min(J_B)$ to $\max (J_B)$
in order to pick a suitable $(t,p)\in D\cap M$ that is compatible with
$(s^\dagger, q)$. Recall that the main requirement, once we know
that $(t,p)\in \mathbf{L}\cap M$, is to have that 
 $\delta_j \in a_k$ for each $j\in J_B$
  and each $1\leq k\leq \ell$ with $i_j = i_k$, where
 $\Delta_{T_{t,p}} = \langle \delta_1, \ldots, \delta_\ell\rangle$.
We begin with $j_0 = \min(J)$, and we note that $s^\dagger$ forces
that $\dot F_{j_0-1} \in M_0$ is non-empty.  By Lemma
\ref{basiclemma},
 there is an $t_0\in M\cap S$ with $t_0<s^\dagger$ 
that also forces $\dot F_{j_0-1}$ is not empty. By elementarity,
 there is a sequence $\vec \delta_0\in M_0$ such that
 $t_0 \Vdash \vec\delta_0 \in \dot F_{j_0-1}$. 
By definition, $t_0 \Vdash \dot F_{j_0}(\vec \delta_0) $ is
stationary.
 Now we use our assumptions on  $\dot {\mathcal I}$
  in order to find a member of $\dot F_{j_0}(\vec\delta_0)$
  that  is in $a_{j_0}$.
This next step can seem  a bit like  sleight of hand.  We have
that $t_0\Vdash \dot F_{j_0}(\vec \delta_0)$ is stationary,
and so there is an extension (in $M_0$) of $t_0$ and an infinite
set $a$ that is forced to be contained in $\dot F_{j_0}(\vec \delta_0)$
and to be a member of $\dot {\mathcal I}$. However,
$t_0$ may be
incomparable with $s_{i_{j_0}}$ and so $a$ is of no help in
choosing a suitable element of  $a_{j_0}$.  The solution is to use that 
$S$ is coherent. Let $g$ be a generic filter for $S$ with 
 $s^\dagger\in g$. Since $S$ is coherent, the collection
   $s_{i_{j_0}}\oplus g = \{ s\in S : (\exists t\in g)~~ s < (s_{i_{j_0}} \oplus
     t )\} $ is also a generic filter for $S$ since it is an $\omega_1$-branch.
  The ideal $\mathcal I(s_{i_{j_0}}) $ we get by interpreting
  the name $\dot {\mathcal I}$ using the filter $s_{i_{j_0}}\oplus B$,
  is a P-ideal satisfying that $[E]^{\aleph_0}\cap \mathcal I(s_{i_{j_0}})$
  is non-empty for all stationary sets $E$. Also, the set
   $E = \val_{g}(\dot F_{j_0}(\vec \delta_0))$  is a stationary set. 
   By elementarity, there is an infinite set $a\in M_0$ 
      such that $a\in [E]^{\aleph_0}$ and $a\in \mathcal I(s_{i_{j_0}})$.
Again by elementarity, and Lemma \ref{basiclemma},
there is   a condition $t_1\in M\cap g$  extending
   $t_0$ and
 satisfying   that $t_1 \Vdash a \subset 
   \dot F_{j_0}(\vec\delta_0)$ 
   and $s_{i_{j_0}}\oplus t_1 \Vdash a \in \dot {\mathcal I}$.
 Let us note that $a\subset a_k$ for each $1\leq k\leq m$
  such that $i_{j_0} = i_k$. Therefore, we may choose
    a $\delta_{j_0}
  \in a\cap \bigcap \{ a_k : i_k = i_{j_0}\}$. Next choose 
  any sequence $\vec\delta_1\in M_0$
  such that, by further extending $t_1$, we have that
   $t_1 \Vdash \vec\delta_1\in \dot F_{j_0}$ and 
  witnesses that $\delta_{j_0}\in \dot F_{j_0}(\vec \delta_0)$.
  
  We proceed in the same way to choose $\delta_{j_1}$ and
  an extension $t_2$ of $t_1$ so that $\delta_{j_1} \in a_k$
  for each $k$ with $i_k = i_{j_1}$ and so that $t_2$ forces
  that there is a $\vec\delta_2\in \dot F_{j_2}$ witnessing
  that $\delta_{j_1}\in \dot F_{j_2}(\vec\delta_1)$. Proceeding
  in this way we succeed in choosing $t_{\bar \ell}$ in 
   $M_0$ with $t_{\bar\ell}\subset
   s^\dagger$ and a sequence $\vec\delta_{\bar\ell}$
   satisfying that there is a $p\in M_0$ 
   such that $(t_{\bar\ell}, p) \in \mathbf {L}$,
    $\Delta_{T_{t,p}} = \langle \delta_1, \ldots , \delta_\ell\rangle
    = \vec\delta_{\bar\ell}$, 
   and that
     $\delta_{j_n} \in a_k$ 
      for each $1\leq n \leq \bar \ell$ and
     $1\leq k\leq \ell$ such that 
      $i_k = i_{j_n}$.  Of course this means that $(t_{\bar \ell}, p)\in D\cap M$
      and $(t_{\bar \ell}, p) \not\perp (s^\dagger, q)$ as required.
 \end{proof}

\section{on $\mathbf{PPI}^+$}

This first result is a reformulation of a classic result of
Sapirovskii.

\begin{lemma} Assume that $X$ is a\label{itsomega1}
sequentially compact non-compact space. 
Then either $X$
 has a countable  subset with non-compact  closure
or $X$   has  an $\aleph_1$-sized subset $E$
and an open set $\mathcal W$ containing the sequential closure
of $E$ and such that $E$ 
has no complete accumulation point in $\mathcal W$.
\end{lemma}

\begin{proof}
We may as well assume that countable subsets of $X$ have 
compact closure. Since $X$ is sequentially compact
and not compact, it is not Lindel\"of.
Let $\mathcal U$ be any open cover of $X$ that has no
countable   subcover and satisfies that the closure of each
member of $\mathcal U$ is contained in some other member
of $\mathcal U$.  We begin an inductive construction by choosing 
any countable subset   $\mathcal U_0$ of $  \mathcal U$ and
any point $x_0 \in X\setminus \bigcup \mathcal U_0$.  Suppose
 $\lambda < \omega_1$, and that we have chosen,
 for each $\alpha < \lambda$, a countable collection
  $\mathcal U_\alpha\subset \mathcal U$ and a point $x_\alpha
  \in X \setminus \bigcup \mathcal U_\alpha$,
  so that $\overline{\{x_\beta: \beta <\alpha\}}\cup
\bigcup_{\beta<\alpha} \mathcal U_\beta \subset  \mathcal U_\alpha$.
Since $\mathcal U$ has no countable subcover, this induction continues
for $\omega_1$-many steps. We let $E$ be the 
sequence $\{x_\beta : \beta\in \omega_1\}$ and let
 $\mathcal W$
 be the union of the collection $\bigcup\{ \mathcal U_\alpha :
 \alpha \in \omega_1\}$. 
By construction we have that the closure
of every countable subset of $E$ is
contained in $\mathcal W$. But also, for each $y\in \mathcal W$, 
we have
that there is an $\alpha\in \omega_1$ such that
 $y\in \bigcup\mathcal U_\alpha$ while
 $\bigcup\mathcal U_\alpha\cap E = \{ x_\beta : \beta < \alpha\}$ is 
countable.
\end{proof}

For the remainder of the section we have an
 $S$-name of a sequentially compact non-compact
  space $\dot X$ which we may assume has base set $\theta$.
According to Lemma \ref{itsomega1}, we will assume 
that either 
 $\omega\subset \theta$ is forced (by 1) to be dense in $X$,
or,  that
the sequential closure
of the points $\omega_1\subset \theta$ are forced (by 1)
to have no
complete accumulation point in some open neighborhood
 $\mathcal W$. In particular then, the sequential closure of
 $\omega_1$ itself
contains no complete accumulation point of $\omega_1$.

In the first  case, our application
of PFA(S) will be simplified if we use
the method sometimes called the 
\textit{cardinal collapsing trick\/}.
This is to show that we may again assume that we have
an uncountable set, denoted $\omega_1$, so that the 
sequential closure is contained in an open set $\mathcal W$
in which $\omega_1$ has no complete accumulation point.
It will be easier to remember if we call this the 
separable case. 
The simple countably closed poset $2^{<\omega_1}$ is $S$-preserving.
We will work, for the separable case,
 in the forcing extension by $2^{<\omega_1}$ -- a model in which CH
holds. Just as we have in Lemma \ref{itsomega1}, we would like
to show that there is an uncountable set $E$ and an open set 
$\mathcal W$
so that the sequential closure of $E$ is contained in 
$\mathcal W$ while
having no complete accumulation points in $\mathcal W$. 
We certainly 
have that the forcing $2^{<\omega_1}$ preserves that $\dot X$
is forced by $S$ to be sequentially compact and not compact. 
We briefly work in the forcing extension by $2^{<\omega_1}\times S$.
Let $X$ denote the space obtained from the name $\dot X$.
If the base set $\theta$ for $ X$
is equal to $\mathfrak c$ then, since it is forced to be countably
compact, it is forced that $ X$ has an uncountable set with no
complete accumulation point at all.  On the other hand if
$ X$ has cardinality greater than $\mathfrak c$, we can fix
any point $ z$ of $X$ that is  not  in the sequential closure of 
 $\omega$. Since $ X$ is regular and $\omega$ is dense,
the point $ z$ has 
character $\omega_1$. Let $\{  W_\alpha : \alpha \in \omega_1\}$ 
enumerate a neighborhood base for $z$ satisfying that the closure
of $W_{\alpha+1}$ is contained in $W_\alpha$ for each $\alpha\in
\omega_1$. For each $\alpha$, we may choose a point $x_\alpha$ from
the sequential closure of $\omega$ so that $x_\alpha$ is in $W_\beta$
for all $\beta \leq \alpha$. Now we have that the uncountable set $E =
\{ x_\alpha : \alpha\in \omega_1\}$ satisfies that its sequential
closure is contained in the open set $\mathcal W
 = X\setminus \{z\}$ and has no
complete accumulation point in $\mathcal W$. 

\bigskip

Next we choose  an assignment of $S$-names of neighborhoods $\{ \dot
U(x,n) :   x\in \theta, n\in \omega \}$, each of which is forced
to have closure contained in $\mathcal W$.
 We may assume that 1 forces that these 
 are regular  descending and that $\omega_1\cap \dot U(x,0)$ is countable
for all $x\in \theta$. These are chosen in the generic extension
by $2^{<\omega_1}$ in the separable case.
If we are also assuming that $\dot X$ is forced to be first countable,
then we assume that $\{ \dot U(x,n) : n\in \omega\}$ is forced
to  form a neighborhood  base for $x$.

\subsection{the sequential structure}

Since $S$ is ccc, it follows that if $\{ \dot x_n : n \in \omega\}$ is
a sequence of $S$-names and $1\forces \dot x_n \in X$ for each $n$,
then there is an infinite $L\subset \omega$ such that $1\forces
\{\dot x_n : n \in L\}$ is a converging sequence in $X$. To see this, 
 recursively choose a
 mod finite descending sequence $\{ L_\alpha :\alpha \in \gamma\}$
 and conditions $\{ s_\alpha : \alpha\in \gamma\}$ satisfying that
  $s_\alpha$ forces that $\{ \dot x_n : n\in L_\beta \}$ (for
   $\beta < \alpha$) is not converging, while
    $\{ \dot x_n : n\in L_\alpha \}$ is. Since the family
     $\{ s_\alpha : \alpha \in \gamma\}$ is an antichain, this process
     must end.

\begin{definition} Say that a sequence $\{ \dot x_n : n\in L \}$ is an
  $S$-converging sequence in $\dot X$
 providing $1\Vdash \{\dot x_n : n\in L\}$
is a converging sequence (which includes, for example, constant
sequences).  
\end{definition}

There is a well-known space in the study of sequential spaces, namely
the space $S_\omega$ from \cite{ArhFran}. This is the strongest
sequential topology on the set of finite sequences of integers,
$\omega^{<\omega}$, in which, for each $t\in \omega^{<\omega}$, the
set of immediate successors, $\{ t^\frown n : n\in \omega\}$,
converges to $t$. If $T$ is any subtree  of $\omega^{<\omega}$, we will
consider $T$ to be topologized as a subspace of $S_\omega$.  As
usual, for $t\in T$, $T_t$ will denote the subtree with root $t$ and
consisting of all $t'\in T$ which are comparable with $t$. 
Also for $t\in T$, let $T_t^+$ denote the tree 
$\{ t' \in \omega^{<\omega} : t^\frown t' \in T_t\}$
i.e. the canonically isomorphic tree with root $\emptyset$.

Of particular use will be those 
$T \subset \omega^{<\omega}$ that are
well-founded (that is, contain  no infinite branch). 
 Let $\mathbf{WF}$ denote those downward closed
well-founded  trees $T$ with the property that every branching
node has a full set of immediate successors.
Such a tree will have a root, $\operatorname{root}(T)$
 (which need not be the root of $\omega^{<\omega}$)
which is either the minimal branching node or,
 if there are no branching nodes, the maximum member
of $T$.
When discussing the topology on $T\in \mathbf{WF}$ we ignore
the nodes strictly below the root of $T$.
The meaning
of the rank of $T$ will really be the rank of $T_t$ where $t$ is the
root of $T$.  We use $\operatorname{rk}(T)$ to denote the ordinal $\alpha\in
\omega_1$ which is the rank of $T$. If $t\in T$ is a maximal node, 
then $\operatorname{rk}(T_t) = 0$, and if
$\operatorname{root}(T)\subset t\in T$, then 
 $\operatorname{rk}(T_t) = \sup \{ \operatorname{rk}(T_{t'})+1 
 : t < t'\in T_t\}$.
We let $\mathbf{WF}(\alpha) = \{ T\in \mathbf{WF} :
\operatorname{rk}(T) \leq  \alpha\}$ and
 $\mathbf{WF}({<}\alpha) = \bigcup_{\beta<\alpha} \mathbf{WF}(\beta)$.

If we have a Hausdorff space $X$ on a base set
containing the set $ \omega_1$ and we have a point 
$x$ in the sequential closure of $\omega_1$, then 
there is a $T\in \mathbf{WF}$ and a function $y$
from $\max(T)$ into $\omega_1$ such that there is a continuous
extension of $y$ to all of $T$ 
such that $y(\operatorname{root}(T)) = x$
 (it does not matter what value  $y(t)$ takes for
 $t<\operatorname{root}(T)$). 
Since our space $\dot X$ is forced to be sequentially compact, we will
be working with points in the sequential closure of $\omega_1$. 
In fact, we will only work with such function pairs
 $y, T$ that are forced by 1 to extend continuously to all of $T$.
The goal is to try to make choices of points in $\dot X$ that
are, in a strong sense, not dependent on the generic filter for $S$.

 For each $\alpha\in \omega_1$, let $\mathbf{Y}_\alpha$
denote the set of all functions $y$ into $\omega_1$
where $\dom(y)$ is the set of all maximal nodes of some $T\in 
\mathbf{WF}(\alpha)$. 
We put $y\in \mathbf{Y}_\alpha$
 in $Y_\alpha$ providing $1$ forces
that $y$ extends continuously  to all of $T_y$ as a function into
$\dot X$. We let $Y=\bigcup_{\alpha\in \omega_1} Y_\alpha$ and
for $y\in Y$, we will
abuse notation by letting $y$ also denote the name of the
unique continuous extension of $y$ to all of
 $\{ t\in T_y : \operatorname{root}(T) \subset t\}$. 
More precisely, if needed, for each $t\in T_y\setminus \max(T)$ with 
 $\operatorname{root}(T)\subset t$, $y(t)$ can be used to denote 
the name  that has the form  $\{ (s,\xi_s) : s\in S_{\gamma}\}$ 
where $\gamma$ is minimal such that each $s$ in $S_\gamma$
decides the value of $y(t')$ for each $t\leq t'\in T_y$
 in the continuous extension of $y$
and, of course, $s$ forces $y(t) = \xi_s\in \theta$  for each $s\in S_\gamma$.
The minimality of $\gamma$ makes this choice canonical. 
Thus for $y\in Y$ and $\operatorname{root}(T_y)\subset t\notin
\max(T_y)$, 
 the sequence $\{ y({t^\frown n}) : n\in \omega   \}$ is an $S$-converging
  sequence that is forced to converge to $y(t)$. Note also that if
$y\in Y$ and $t\in T_y$, then $y\restriction (T_y)_t \in Y$. 

\begin{definition} Say that
 $y_1$ and $y_2$ in ${Y}$ are equivalent, denoted $y_1\approx
 y_2$, providing $T_{y_1}^+ = T_{y_2}^+$, and for each maximal $t\in
 T_{y_1}^+$, $y_1(\operatorname{root}(T_{y_1})^\frown t) $
equals  $y_2(\operatorname{root}(T_{y_2})^\frown t) $. 
\end{definition}

Clearly if $y_1\approx y_2$, then $y_1(\operatorname{root}(T_{y_1}))$
is the same name as  $y_2(\operatorname{root}(T_{y_2}))$. 
Now that we have identified our structure
 $Y$  we extend the notion to define a closure
operator on any given finite power of $Y$ which will help
us understand points in the sequential closure of $\omega_1$ in $\dot
X$. If $y\in Y$, we use $e(y)$ as an alternate notation
for $y(\operatorname{root}(T_y))$. Similarly, if $\vec y\in Y^n$ (for
some $n\in \omega$), we will use $e(\vec y)$ to denote the point
 $\langle e(\vec y_0), e(\vec y_1), \ldots, e(\vec y_{n-1})\rangle$.

\begin{definition}
 For\label{seql}
 each integer $n>0$, and subset $B$ of $Y^n$ we similarly define
 the hierarchy $\{ B^{(\alpha)} : \alpha\in \omega_1\}$
by recursion.  In addition, we again (recursively) view
each $ \vec b\in B^{(\alpha)}$ as naming a point in $X^n$. The
 set $B$ will equal $B^{(0)}$. Naturally the point $e(\vec b)$
named
 is the point of $X^n$ named coordinatewise by $\vec b$.
 
 For limit $\alpha$, $B^{(\alpha)}$
 (which could also be denoted as $B^{(<\alpha)}$) will equal
  $\bigcup_{\beta<\alpha} B^{(\beta)}$. The members of $B^{(\alpha+1)}$
  for any $\alpha$, will consist of the union of 
  $B^{(\alpha)}$ together with all
  those $\vec b = \langle y_i : i\in n\rangle \in (Y_{\alpha+1})^n$ such
  that there is  
a sequence $\langle \vec b_k : k\in \omega \rangle$ so that
\begin{enumerate}
\item for each $k\in \omega$, $\vec b_k$ is a member  of $B^{(\alpha)}\cap
  (Y_\alpha)^n$, 
\item for each $i\in n$ and $k\in \omega$, $(\vec b_k)_i\in Y$ is equivalent to
 $y_i\restriction (T_{y_i})_{t_i^\frown k}$, where $t_i$ is the root
 of $T_{y_i}$.
\end{enumerate}
When $\vec b$ is constructed from a sequence $\{ \vec b_k : k\in \omega\}$
as in this construction, we can abbreviate this by saying that
 $\{ \vec b_k : k\in \omega\}$ $Y$-converges to $\vec b$. Also if 
  we say that $\{ \vec b_k : k\in L\}$ Y-converges to $\vec b$ for some
  infinite set $L$, we just mean by a simple re-enumeration of
   $\{ \vec b_k : k\in L\}$.
\end{definition}

 For $n>1$ we may view $Y^n$ as an $S$-sequential structure
and for any $A\subset Y^n$, we say that $A^{(\omega_1)}$ is the
sequential closure and is sequentially closed.
 Notice that this   $S$-sequential structure on $Y^n$ is defined in the
ground model.

The next lemma should be obvious.

\begin{lemma} For each $A\subset Y$, $1$ forces that $e[
A^{(\omega_1)}]$
  is a sequentially compact subset of $X$.
\end{lemma}

\begin{definition} For each $S$-name $\dot A$ and $s\forces \dot A\subset
  Y^n$,  we define 
the $S$-name  $(\dot A)^{(\omega_1)}$ according to the
  property that for each $s<t$ and $t\forces \vec y\in (\dot
  A)^{(<\omega_1)}$, there is a countable $B\subset Y^n$ such that 
 $t\forces B\subset \dot A$ and $\vec y\in B^{(<\omega_1)}$. 
\end{definition}

 For an $S$-name $\dot A$ and $s\forces \dot A\subset Y^n$, we will
also  interpret 
 $e[(\dot A)^{(\omega_1)}]$
in the forcing extension in the natural way
 as a subset of $\dot X^n$.  This may need some further
clarification.
 
\begin{lemma}
Suppose that   $\vec y$ is a member of $B^{(\alpha)}$ for some
$B\subset Y^n$ and some $\alpha<\omega_1$.
Also suppose that $\{ s_i : i <\ell \}\subset S$ and that $\dot W$ is
an S-name for a  
neighborhood\label{cohere} of
 $e({\vec y}\,)$. Then there is a $\vec b \in B$ such that for each $i<\ell$, 
 there is an $s_i' \supset s_i$ forcing that $e({\vec b}\,)\in \dot W$.
\end{lemma}

\begin{proof}
  We may suppose that $B$ is a countable set and we may proceed
  by induction on $\alpha$.  Let $y_i $ be equal to $\vec y_i$ for
   $i<n$, and let $t_i$ denote the root of $T_{y_i}$.
  By the definition of $B^{(\alpha)}$, 
  there is  a   sequence 
  $\langle \vec y_k : k\in \omega\rangle$
   with the property that $\vec y_k\in B^{(<\alpha)}$ for each $k$,
   such that, for each $i<n$ and each $k\in \omega$, 
   $(\vec y_i)\restriction (T_{y_i})_{t_i^\frown k}$ is
   equivalent to $(\vec b_k)_i$.
   For each $i<\ell$, choose $s_i' \supset s_i$
  so that $s_i'$ forces a value, $W_i$,  on  
$\dot W\cap e[ 
B\cup\{\vec y_k  : k\in \omega\}]$.     
   Since this sequence, $\{ e(\vec y_k ) : k\in \omega\}$ 
   is assumed to be $S$-converging to a point
   in $W_i$, there is a $k$ such that the point $e(\vec y_k)$
   is in each $W_i$. Of course the result now follows by the induction
   hypothesis.
   \end{proof}

Note that the
 members of ${Y}_0$ have  a singleton domain
and for
each  $\alpha\in \omega_1$, let $e^{-1}(\alpha)$ denote the 
member
of $Y_0$ that sends the minimal tree to the singleton $\alpha$.
Our assumption that $\omega_1$ has no complete accumulation points
in its sequential closure implies that no point is a member
of every member of the family 
$\{~ \left(\{e^{-1}(\alpha):\alpha > \delta \}\right)^{(\omega_1)} :
 \delta\in \omega_1\}$.  That is, 
 this family is a 
 free filter of $S$-sequentially
 closed subsets of $Y$. By Zorn's Lemma, we
can extend it to a  maximal free filter, $\mathcal F_0$,
of  $S$-sequentially closed subsets of $Y$.

\subsection{A new idea in PFA(S)}

Now we discuss again the special forcing
properties that a  coherent Souslin tree will have.
 Assume that $g$ is (the) generic filter on $S$ viewed as a
cofinal branch.  For each $s\in S$, $o(s)$ is  the level
(order-type of domain) of $s$ in  $S$. For any $t\in S$, define
$s \oplus t$ to be the function $s\cup t\restriction[o(s),o(t))$. 
Of course when $o(t) \leq o(s)$, $s\oplus t$ is simply $s$.
 One of the
properties of $S$ ensures that $s \oplus t\in S$ for all $s,t\in S$. We
similarly define $s \oplus g$ to be the branch $\{ s \oplus t : t\in
g\}$.

\begin{definition} Let $bS$ denote the set of $\omega_1$-branches of
  $S$.
\end{definition}

\begin{lemma} In the extension $V[g]$, $bS = \{ s \oplus g : s\in
  S\}$. Furthermore, for each $s\in S$, $V[s \oplus g] = V[g]$. 
\end{lemma}

The filter $\mathcal F_0$ may not generate a maximal filter in the
extension 
$V[g]$ and so we will have to extend it. Looking ahead to the PFA(S)
step, we would like (but probably can't have) this (name of) 
extension to give
the same filter in $V[s \oplus g]$ as it does in $V[g]$. We adopt a new
approach. We will define a filter (of $S$-sequentially closed) subsets
of the product structure $Y^{bS}$. We try to make this filter somehow
symmetric. 

We introduce some notational conventions.  Let $S^{<\omega}$ denote
the set of finite tuples 
$\langle s_i : i< n\rangle$ for which there is a
$\delta$ such that each $s_i\in S_\delta$. Our convention will be that
they are distinct elements. We let $\Pi_{\langle s_i : i<n\rangle}$ denote the
projection from $Y^{bS}$ to $Y^n$ (which we identify with 
 the product $Y^{ \{ s_i \oplus g : i< n\}}$). 

\begin{definition} Suppose that  $\dot A$ is an $S$-name of a subset
  of $Y^n$ for some $n$, in particular, that  some $s$ forces
  this. Let  
 $s'$ be any other member of $S$ with $o(s')=o(s)$. We define a new
 name $\dot A^{s}_{s'}$ (the $(s,s')$-transfer perhaps) which is
 defined by the property that for all $\langle y_i \rangle_{i<n}\in
 Y^n$ and  $s<t\in S$ such that $t\forces \langle y_i \rangle_{i<n}
 \in \dot A$, we have that $s' \oplus t \forces \langle y_i
 \rangle_{i<n} \in \dot A^s_{s'}$. 
\end{definition}

\begin{lemma} For any generic $g\subset S$,
 $\val_{s \oplus g} (\dot A) = \val_{s' \oplus g} (\dot A^s_{s'})$.
\end{lemma} 

\begin{theorem} There is a  family $\mathcal F$ = 
$\{ 
(s^\alpha,\{s^\alpha_i : i<n_\alpha\}, \dot F_\alpha ) : \alpha \in
  \lambda\}$ where, 
\begin{enumerate}
\item for each $\alpha\in \lambda$, $
\{ s^\alpha_i :  i<
  n_\alpha\}\in S^{<\omega}$, $s^\alpha\in S$, $ o(s^\alpha_0)\leq o(s^\alpha)$,  
\item $\dot F_\alpha$ is an $S$-name such that $s^\alpha\Vdash \dot
  F_\alpha = (\dot F_\alpha)^{(\omega_1)} \subset Y^{n_\alpha}$ 
\item for each $s\in S$ and $F\in \mathcal F_0$, $(s,\{s\}, \check F)\in 
 \mathcal F$,
\item for each $s\in S_{o(s^\alpha)}$, $(s,\{s^\alpha_i :i<
  n_\alpha\}, (\dot F_\alpha)^{s^\alpha}_{s})\in \mathcal F$,
\item for each generic\label{previousitem}
 $g\subset S$, the family
$ \{ \Pi_{\langle s^\alpha_i : i<n_\alpha\rangle}^{-1}( \val_{g}(\dot F_\alpha)) : 
  s^\alpha \in g\}$ is finitely directed;
we let $\dot {\mathcal  F}_1$ be the $S$-name for the filter base it
generates,
\item for each generic $g\subset S$ and each $\langle  s_i : i< n\rangle\in 
 S^{<\omega}$, the family \\
  $\{ \val_g (\dot F_\alpha) : s^\alpha\in g \ \mbox{and}\ 
 \{ s_i \oplus g : i< n\} = \{ s^\alpha_i \oplus g : i<n_\alpha\} \}$ is
 a maximal filter on the family of $S$-sequentially closed subsets of
 $Y^n$.
\end{enumerate}
\end{theorem}

\begin{proof} 
Straightforward recursion or Zorn's Lemma argument over the family 
 of ``symmetric'' filters (those satisfying (1)-(5)).
\end{proof}

\begin{definition}
 For any $\langle s_i : i<\ell\rangle\in S^{<\omega}$, let 
  $\dot {\mathcal F}_{\langle s_i:i<\ell\rangle}$ denote the filter on 
    $Y^{\ell}$ induced by $ \Pi_{\langle s_i : i<\ell\rangle}(\dot{\mathcal F}_1)$.
\end{definition}

\begin{definition} Let $\mathcal A$ denote the family\label{needA}
 of all
$(s,\langle s_i : i< \ell\rangle, \dot A)$ satisfying that  $o(s)\geq
o(s_0)$,  
$\langle s_i :
i< \ell\rangle \in S^{<\omega}$, and 
 $s\Vdash 
 \dot A \in \dot
 {\mathcal F}_{\langle s_i:i<\ell\rangle}^+$. As usual, 
for a family $\mathcal G$ of set,
 $\mathcal  G^+$ denotes the family of sets that meet each member of
 $\mathcal G$.
\end{definition}

\begin{lemma} For each $(s,\langle s_i: i<n\rangle, \dot A)\in \mathcal A$, 
the object  $(s,\langle s_i:i<n\rangle , \dot
A^{(\omega_1)})$ is in the list $\mathcal F$.
\end{lemma}

In this next Lemma it is crucial that there are no dots on the sequence
 $\langle y^M(s) : s\in S_\delta\rangle$.  The significance of there
 being no dots is that, regardless of the generic $g\subset S$, 
 we will have that $e_g(y^M(g\cap M))$ is the limit of the same
 sequence from within $M$.

\begin{lemma} Suppose that
 $M\prec H(\kappa)$ (for suitably big $\kappa$) is a
  countable elementary submodel\label{Melement} 
  containing $Y,  \mathcal
  A$. Let $M\cap \omega_1 = \delta$.
There is a sequence $\langle  y^M(s) : s\in S_\delta\rangle$ such
that for every $(\bar s,\{s_i : i<n\}, \dot A)\in \mathcal A\cap M$,
and every $s \in S_\delta$ with $\bar s < s$, there is a
 $B\subset Y^n\cap M$ such that 
$ \langle  y^M(s_i \oplus s) : i< n\rangle \in
B^{(\delta+1)}$ and
$s\forces B\subset \dot A$.
\end{lemma}

\begin{proof}
Let $\{ (s^m,\{s^m_i : i<n_m\} , \dot A_m ):  m\in \omega\}$ enumerate
the family $\mathcal A\cap M$. Also, fix an enumeration, $\{
s^\delta_m : m \in \omega\}$, of $S_\delta$. Let $\{\alpha_m : m\in
\omega\}$ be an increasing cofinal sequence in $\delta$. 
At stage $m$, we let $\beta_m$ be large enough
so that $s_0^\delta
\restriction [\beta_m, \delta) = s^\delta_i \restriction [\beta_m,
\delta)$ for  
all $i< m$. Replace
 the list $\{ (s^j , \langle s^j_i : i< \ell_j\rangle, \dot A_j) :
 j<m\}$
with (abuse of notation of course) $\{ 
(s^j , \langle s^j_i : i< \ell_j\rangle, \dot A_j) : j<L_m\}$ 
so that for all $i,j<m$ with $s^j< s^\delta_i$ (from original list) the new list
includes $(s^\delta_i\restriction\beta_m , \langle s^j_i \oplus
s^\delta_i \restriction \beta_m :   
i<\ell_j\rangle, \dot A_j)$; and so that for all $j<L_m$ in the new list
 $s^j$ and $s^j_i$ are all in $ S_{\beta_m}$. Nothing else is added to the new list
 (in particular, the new list is contained in $\mathcal A\cap M$).
 
 Now we have $s^j_k  \oplus s^\delta_0 = s^j_k \oplus s^\delta_i$ for
 all $i<m$, $j<L_m$ 
 and suitable $k<\ell_j$. Also, whenever $s^j< s^\delta_i$
 ($s^\delta_i$  
 for $i<m$ is unique) we have that
 $s^\delta_i\forces \dot A_j\in \dot {\mathcal F}^+_{\langle s^j_i : i<
   \ell_j\rangle}$; 
 and so we also have that
  $s^\delta_0\forces (\dot A_j)^{s_i^\delta}_{s^\delta_0} 
  \in \dot {\mathcal F}^+_{\langle s^j_i : i< \ell_j\rangle}$ 
  (because they are essentially the same sets).

  Let $\Sigma = \{\sigma_k : k\in K\}$ lex-enumerate $
  \{ (s^j_i
\oplus s^\delta_0)\restriction \beta_m : j<L_m, i<\ell_j\}\in M$. 
Let
   $\Pi_\Sigma $ be the projection map from $Y^{bS}$ onto $Y^\Sigma$.
 Let  $\Pi^\Sigma_{\langle s^j_i : i< \ell_j\rangle}$ be defined
by the equation $\Pi^\Sigma_{\langle s^j_i : i< \ell_j\rangle} \circ \Pi_\Sigma
=
\Pi_{\langle s^j_i : i< \ell\rangle}$. 
We consider the filter (name) $\dot {\mathcal F}_{\Sigma}$. For each
$j<L_m$ and $i<m$ such that $s^j<s^\delta_i$,
 it is forced
 by $s^\delta_0$ that the set $\left(
  \Pi_{\langle s^j_i 
    : i< \ell_j\rangle}^\Sigma\right)^{-1}((\dot
A_j)^{s^\delta_i}_{s^\delta_0})^{(<\omega_1)})$ is 
a member of $\dot {\mathcal F}_{\Sigma}$ and all are in $M$. Select
any
 $\vec y_m \in M\cap Y^\Sigma$ with the property that
 $\Pi^\Sigma_{\langle s^j_i : i< \ell_j\rangle}(\vec y_m)\in \dot
 A_j^{(<\omega_1)}$ for all $j< L_m$. 
 Choose a sequence $\{ B_j : j<L_m\}$ of countable subsets of 
  $Y^{<\omega}$ (in fact $B_j\subset Y^{\ell_j}$)
  which are in $M$ and satisfy that, for each $j< L_m$,
 $s^\delta_0\forces B_j\subset  
  (\dot A_j)^{s_i}_{s_0}$ (where $s^j<s^\delta_i$), and so that
$\Pi^\Sigma_{\langle s^j_i : i< \ell_j\rangle}(\vec y_m)\in
B_j^{(<\delta)}$.
Note that if $s^j<s^\delta_i$, then $s^\delta_i \forces
 B_j\subset \dot A_j$.
 
    If we now return to the ``original'' list, we 
have that for all $i,j<m$ and $s^j<s^\delta_i$, 
 $s^\delta_i \forces 
    \vec y_m\restriction 
   \langle s^j_k\oplus s^\delta_0 \restriction \beta_m : k<\ell_j\rangle \in 
    B_j^{(<\delta)}$.

Now suppose 
 we have so chosen $\vec y_m$ for each $m\in \omega$.
 We assert the existence of an infinite set $L\subset \omega$ with
the property that for all $j,i\in \omega$, $s^\delta_i$ forces that
the sequence $\{  \vec y_m(s^\delta_i\restriction \alpha_m\oplus 
s^\delta_j) : 
m\in L\}$ is defined and $S$-converging on a cofinite set.
For each $i$, $ y^M(s^\delta_i)$ is the $S$-name in $Y$ 
which is   equal to the limit  of this $S$-converging sequence.
\end{proof}

\subsection{$S$-preserving proper forcing}

Now we are ready to define our poset $\mathcal P$.  Recall that
we have a fixed assignment $\{ \dot U(x,n) : x\in \theta, n\in \omega\}$
of $S$-names of  neighborhoods (regular descending for each $x$).

\begin{definition}
 A condition $p\in \mathcal P$ consists of $({\mathcal M}_p, S_p, m_p)$
 where $\mathcal M_p$ is a finite $\in$-chain of countable elementary 
submodels of 
 $(H(\kappa), \{ \dot U(x,m): x\in \theta, m\in \omega\})
$ for some suitable $\kappa$. We let $M_p$
denote the maximal
 element of $ {\mathcal M}_p$ and let
 $\delta_p = M_p\cap \omega_1$. We require that 
$m_p$ is a positive integer and
 $S_p$ is a finite subset of $S_{\delta_p}$. 
 For $s\in S_p$
  and $M\in \mathcal M_p$, we use both 
$s\restriction M$ and $s\cap M$ to
  denote $s\restriction (M\cap \omega_1)$.
We require that the sequence $\{ y^M(s) : s\in S_{M\cap \omega_1}\}$
is in each $M'$
 whenever $M\in M'$ are both in $\mathcal M_p$. This can be made
 automatic if we use a fixed well-ordering of $H(\mathfrak c)$ and
 define
 $\{ y^M(s) : s\in S_{M\cap \omega_1}\}$ to be the minimal sequence
satisfying Lemma \ref{Melement}.

 It is helpful to simultaneously think of $S_p$ as inducing a finite
 subtree,
 $S_p^\downarrow$,
 of $S$ equal to $\{ s\restriction M : s\in S_p, \ \mbox{and}\ 
  M\in \mathcal M_p\}$.

  For each $s\in S_p$ and each $M\in {\mathcal M}_p\setminus M_p$
   we define an $S$-name
   $\dot W_p(s\restriction M)$ of a neighborhood of $e(y^M(s\restriction M))$.
   It is defined as the name of the
 intersection of all sets of the form $\dot U(e(y^{M'}(
s'\restriction M')),m_p)$ 
   where $s'\in S_p$, $M'\in \mathcal M_p\cap M_p$, and
$s\restriction M\subset s'\restriction M'$
    and $e(
    y^M(s\restriction M)) \in \dot U(e(y^{M'}(s'\restriction
    M')),m_p)$. We adopt 
   the convention that $\dot W_p(s\cap M)$ is all of $X$ if
 $s\cap M\notin S_p^\downarrow$.
  \medskip
  
  The definition of $p< q$ is that $\mathcal M_q\subset \mathcal M_p$,
  $m_q\leq m_p$,  
   $S_q \subset S_p^\downarrow$ 
and   for
   each $s'\in S_p$ 
   and 
  $s\in S_q$,  
   we have that $s'$ forces that
   $e( y^M(s\restriction M))\in \dot W_q(s\restriction M')$
   whenever $M\in \mathcal M_p\setminus \mathcal M_q$
and   $M'$ is the minimal member of $M_q\cap (\mathcal M_q\setminus M)$.

It is a notational convenience, and worth noting,
that we make no requirements on sets of
the form $\dot U(s, m_q)$ for $s\in S_q$.
\end{definition}

Before we develop the important properties of $\mathcal P$ let us
check that

\begin{proposition}
If $\mathcal P$ is $S$-preserving, then PFA(S) implies that
 $S$ forces\label{3.9}
  that $\dot X$
\begin{enumerate}
\item contains a  free $\omega_1$-sequence, and
\item if $\dot X$ is first countable, contains a copy of $\omega_1$.
\end{enumerate}
\end{proposition}

\begin{proof}
Let us first consider the easier non-separable case.

For any condition $q\in \mathcal P$, let $\mathcal M(q)$ denote the
collection of all $M$ such that there exists a $p<q$ such 
that $M\in \mathcal M_p$. 
For each $\beta < \alpha\in \omega_1$, $s\in S_\alpha$,
and $m\in \omega$, let
\begin{align*}
D(\beta,\alpha,s,m)  =& \{ p\in \mathcal P : (\exists \bar s\in S_p )~~ 
s < \bar s, m < m_p, \mbox{and}  \\
& (\exists M\in \mathcal M_p)~ (\beta\in M, \alpha\notin M)\ \  \mbox{or}\\ 
& (\forall M\in \mathcal M(p)) (\beta\in M \Rightarrow \alpha\in M)\}~.
\end{align*}
It is easily shown that each $D(\beta,\alpha,s,m)$ is a dense open
subset of $\mathcal P$.
Consider the family 
 $\mathcal D$ of all such $D(\beta,\alpha,s,m)$, and
let $G$ be a $\mathcal D$-generic filter.
Let  $\mathcal M_G = \{ M : (\exists p\in G)~  M \in
 \mathcal M_{p}\}$ 
and let $C = \{ M\cap \omega_1 : M\in \mathcal M_G\}$.
 Let $g\subset S$ be a
 generic filter.
For each $\gamma\in C$ and $M\in \mathcal M_G$ with $M\cap\omega_1 =
\gamma$,  let $x_\gamma = e(y^{M}(g\cap M))$ (we omit the trivial
proof that there is exactly one such $M$ for each $\gamma\in C$).

 We show that the set 
 $W = \{ x_\gamma : \gamma \in C\}$ contains an uncountable free
 sequence of $X$, and that, if $X$ is first countable, $W$
 is homeomorphic to the
 ordinal $\omega_1$. 
If $\gamma<\delta$ are both in $C$ then $x_\gamma$
and $x_\delta$ are distinct. To see this, let us note that
since  $x_\gamma\in M_{\delta}$ there is a $\beta\in
M_\delta$ such that $U(x_\gamma,0)\cap \omega_1\subset \beta$. 
Also,  the closure of
$U(x_\gamma,1)$ was assumed to be contained in $U(x_\gamma,0)$. 
But now, $x_\delta = e(y^{M_\delta}(g\cap M_\delta))$ was chosen so as
to be in the closure of $e[\dot F\cap M_\delta]$ for each $\dot F\in
M_\delta$. In particular, $x_\delta$ is in the closure of
$\omega_1\setminus \beta$, and so it is not in $U(x_\gamma,1)$. 

We may now define the map $f$ sending $x_\gamma$ to the
 ordinal $o.t.(C\cap \gamma)$  (the order type).
It is certainly 1-to-1 and onto. Let $\{ \xi_n : n\in \omega\}\subset
\omega_1$ be  strictly increasing with supremum $\xi$. For each $n$,
let $f(x_{\gamma_n}) = \xi_n$ and $f(x_\gamma) = \xi$. Fix any $m\in
\omega$,
 set $s=g\cap S_\gamma$ and choose
any $p\in G\cap D(0,\gamma,s,m)$. Since $\gamma\in C$, we may assume
by extending $p$, that there is an $M_\gamma \in \mathcal M_p$ with 
 $M_\gamma\cap \omega_1 = \gamma$. Let $\beta$ be the maximum element
 of $\{ M\cap \omega_1 : M\in \mathcal M_p\cap M_\gamma\}$. We have
that $f(x_\beta)<\xi$ and so there is an $n_0$ such that 
 $\xi_n > f(x_\beta)$ for all $n>n_0$. Now for any $r<p$ with $r\in G$, 
and $M\in \mathcal M_r$ with $\beta < M\cap \omega_1  < \gamma$
 we have that $s\Vdash e(y^M(s\cap M)) \in \dot U(s\cap M_\gamma,
 m)$. From this it follows that $x_{\gamma_n} \in U(x_\gamma,m)$ for
all $n>n_0$. This shows that each limit point of $\{ x_{\gamma_n}
:n\in \omega\}$ is in $U(x_\gamma,m)$. In fact more is true:
 the set $\{ x_\alpha : \beta <\alpha < \gamma\}$ is
contained in $U(x_\gamma,m)$ and, because 
$U(x_\gamma,0)$
is  in $M_{\gamma+1}$, 
 $U(x_\gamma,0)\cap \{x_\alpha : \gamma < \alpha <\omega_1\}$ is
 empty. This shows that in all of $X$,
   $\{ x_\alpha : \beta < \alpha \leq \gamma \}$ and
$\{ x_\alpha  : \gamma<\alpha\}$ have disjoint closures.
Since $\gamma$ was an arbitrary member of $C$, this shows
that the closure of the full initial segment
 $\{ x_\alpha : \alpha \leq \gamma\}$ is disjoint from 
the closure of
$\{ x_\alpha  : \gamma<\alpha\}$.

Now we assume that $X$ is first countable. We have just given
 a proof using sequences that,
since $\omega_1$ is also first countable, 
 the inverse map, $f^{-1}$, is continuous. 
Since $\dot X$ is forced to be Hausdorff, $f^{-1}$ is also a
homeomorphism. 
\bigskip

Now we consider the separable case.
We are working in the forcing extension by $2^{<\omega_1}$.
For each
 $\alpha\in \omega_1$, let $E_\alpha$ denote the dense open subset of
 $2^{<\omega_1}$ of conditions that decide which member of $\dot
 X$ is equal to the chosen ordinal $\alpha$ of the copy of $\omega_1$
that is forced to have no complete accumulation points
 in $\mathcal W$. In particular,
for each $x\in \theta$ and $n\in \omega$,
 let $E(x,n)$ denote the dense open set of conditions in $2^{<\omega_1}$
 that decide
on the value of the name $\dot U(x,n)$ and that decide the countable
set $\dot U(x,n)\cap \omega_1$. Fix a $2^{<\omega_1}$-name, 
 $\dot {\mathcal P}$ for our poset
$\mathcal P$ as defined above.
We are assuming that 
 $2^{<\omega_1}$ forces that  $\dot{\mathcal P}$ 
is proper and $S$-preserving.  By 
\cite{Tadatoshi} (
and see \cite[4.1]{Todorcevic}) it follows that the iteration
 $2^{<\omega_1} * \dot {\mathcal P}$ is proper and $S$-preserving.
 The rest of the proof proceeds
 just as in the non-separable case.
\end{proof}

We prove a kind of density lemma.

\begin{lemma}
If $\mathcal P\in M$ for some countable $M\prec H(\mu)$, then
 for each $p\in \mathcal P$ and each $\alpha \in M\cap \omega_1$, 
such that $M\cap H(\kappa)\in \mathcal M_p$, there is an\label{singleM} 
 $M'\in M$ such that $\alpha\in M'$, 
 $r= (\mathcal M_p\cup\{M'\} , S_p, m_p)\in
\mathcal P$,  and $r<p$.
\end{lemma}

\begin{proof}
Let $M_0 = M\cap H(\kappa)$ and
let $ S_0 = \{ s_i : i< \ell\}$ enumerate the set
 $\{ s \cap M : s\in S_p\}$ in the lexicographic order. In this proof
we adopt the convention that we will enumerate $S_q$ for any condition
$q$ in increasing lexicographic order.  Let $M^\dagger$ be the maximum
element of $\mathcal M_p\cap M$ and set $S^\dagger = \{ s\cap
M^\dagger : s\in S_p\}$. We define $p\restriction M$ to be
 $(\mathcal M_p\cap M, S^\dagger, m_p)$. It is routine to verify that
 $p< p\restriction M$.

 By increasing $\alpha$ we 
may assume that $ s_i \restriction [\alpha, M\cap \omega_1) 
 =  s_j \restriction [\alpha, M\cap \omega_1)$ for all 
$i,j<  \ell$ and that $\omega_1\cap 
\bigcup 
(\mathcal M_p\cap M) < \alpha$.
 For each $i< \ell$, let $\bar s_i =
s_i\restriction\alpha$ and  let
 $\bar S = \{ \bar s_i  : i< \ell\}$.

It is easily checked that $r = (\{M_0\} \cup  (\mathcal M_p \cap M)
, 
S_0, m_p)$ is in $\mathcal P$ and is an extension of $p\restriction M $. 
 Notice that this implies that
 $S_r$ is equal to $\bar S  \oplus s_0 = \{  \bar s_i\oplus  s_0 :
 i<\ell\}$. 

Define the $S$-name $\dot A$ as
\[   \{(s^q_0, \langle { y}^{M_q} (s^q_i) : i< \ell\rangle) :
q< p\restriction M , M_q\cap \mathcal M_q = \mathcal M_p\cap M, 
 m_q = m_p,   S_q = \bar S\oplus s^q_0 \}\]
This set $\dot A$ is a member of $M_0$ and, by virtue of $r$,
 $(s_0, 
\langle { y}^{M_0}(  s_i) : i< \ell\rangle)$ is an element
of $\dot A$. 
We show, by a density and elementary submodel argument, that
we have that
 $  s_0$ forces that $\dot A$ is in
 $\dot {\mathcal F}^+_{\langle \bar s_i:i<\ell \rangle}$.  
First of all, there is a dense subset of $S$ each member
of which decides the statement ``$(\exists F \in \dot{\mathcal
  F}_{\langle \bar s_i : i < \ell\rangle}) ~~(\dot A\cap F = \emptyset)$''.
Since this dense set is in $M_0$, there is an $\dot F\in M_0$ such that
 $s_0 \Vdash \dot F\in 
\dot {\mathcal F}_{\langle \bar s_i : i<
   \ell\rangle}
$, and either $s_0 \Vdash \dot A \in \dot {\mathcal
   F}_{\langle \bar s_i : i< \ell\rangle}$ or $s_0 \Vdash \dot A\cap
 \dot F = \emptyset$. However, it is clear that there is a $\beta\in
 M_0$ such that $s_0\restriction \beta \Vdash \dot F\in 
\dot {\mathcal F}_{\langle \bar s_i : i<
   \ell\rangle}$ and so, by Lemma \ref{Melement}, $\langle
 y^{M_0}(s_i) : i<\ell\rangle \in \dot F$.
It follows
 then that there is an $s\in M_0$ 
below $s_0$ which also forces that 
 $ \dot A$ is in
 $\dot {\mathcal F}_{\langle \bar s_i :i<\ell\rangle}^+$.   

Well, what all this proves is that 
 $(s,\{\bar s_i : i<\ell\}, \dot A)$
is a member of the collection $\mathcal A$ and is in $M\cap H(\kappa)$. 
 Apply Lemma  
\ref{Melement} and select $B\subset M\cap Y^\ell$ so that
 $s_0\Vdash B\subset \dot A$ and $s_0\Vdash
 \langle  y^M( \bar s_i \oplus s_0) : 
  i<\ell\rangle \in B^{(\delta+1)}$. 
   What this actually means
  (see Definition  \ref{seql})
   is that there is a $\langle \vec b_k : k\in \omega\rangle$
of elements of
 $B^{(\delta)}$ which is $S$-converging coordinatewise
to this element.  Now, by Lemma \ref{cohere}, there is a
$\vec b\in B$ satisfying that each $s\in S_p$ forces
that $e(\vec b) $ is in the product neighborhood
 $\dot W_p(s_0) \times \cdots\times\dot W_p(s_\ell)$.  This $\vec b\in B$
 is of course equal to 
$ \langle {y}^{M_q} (s^q_i) : i< \ell\rangle$ for
some $q$ as in the definition of $\dot A$.  It follows
that $ M_q$ is the desired value for $M'$.
\end{proof}

That was a warm-up. What we really should have proven is

\begin{lemma}
 For all $(s,r)\in S\times \mathcal P$ such that $s\notin M_r$, and
  any\label{forproper} $M_0\in \mathcal M_r$ and 
$(s^j,\langle s^j_i : i<\ell_j\rangle, \dot A_j)\in \mathcal A
 \cap M_0$ with $s^j<s$, there is an  $\vec a \in Y^{\ell_j}\cap 
M_0$ such that $s  \forces \vec a \in \dot A_j$ and for each 
$s'\in S_r$ and $i<\ell_j$,
  $s' \forces e(\vec a_i) \in \dot W_r( (s^j_i \oplus s)\cap M_0)$.
\end{lemma}

\begin{proof}
Let $s_0 = s\cap M_0$.
 By definition of $\mathcal A$, we have that $s^j  \forces
 \dot A_j\in \dot {\mathcal F}^+_{
 \langle s^j_i : i<\ell_j\rangle}$.  By Lemma \ref{Melement}, 
 there is a countable $B\subset M_0$ such that
  $s_0\forces B\subset \dot A_j$ and 
   $s_0\forces  \langle y^{M_0}(s^j_i\oplus s_0)   : i<\ell_j\rangle
   \in B^{(\delta_0+1)}$, where $\delta_0=M_0\cap \omega_1$.
 Apply Lemma \ref{cohere} to conclude there is a
$\vec b\in B$ satisfying that each $s\in S_r$ forces
that $e(\vec b) $ is in the product neighborhood
 $\dot W_r(s^j_0\oplus s_0) \times \cdots\times\dot W_r(s^j_{\ell-1}
 \oplus s_0)$.  
 \end{proof}

All we have to do now is to prove that

\begin{theorem}
 The poset\label{ppithm}
 $S\times \mathcal P$ is proper.
\end{theorem}

\begin{proof}
 As usual, we assume that $M\prec H(\mu)$ (for some suitably large $\mu$) is countable,
  and that $M_0 = M\cap H(\kappa)\in {\mathcal M}_p$ for some condition $p$. Let
   $D\in M$ be a dense subset of $S\times \mathcal P$ and assume that
   $(s^\dagger,r)\in D$. 
We may assume that
there is some elementary submodel $M_r$
including $r$,
that $s^\dagger \notin M_r$,
 and that $s^\dagger\cap M_r\in S_r$. This means that for all
$s' \in \tilde M$ and $x\in M_r\cap \theta$, $s'\oplus s^\dagger$
forces 
a value on $\dot U(x,m_r) \cap M_r$. Let $\langle M_i : i\in
\bar\ell\rangle $ enumerate $\mathcal M_r\setminus M$ in increasing
order. By Lemma \ref{singleM} we can assume that 
$\mathcal M_r \cap M$  is not empty, and let
 $M^\dagger $ denote the maximum element. Let $\alpha = M^\dagger\cap \omega_1$
and $\delta_0 = M_0\cap \omega_1$.  We may additionally assume
that
$s\restriction 
[\alpha,\delta_0) = s'\restriction [\alpha,\delta_0)$ for all $s,s'
\in S_r$.

 The plan now is to find $q\in M\cap \mathcal P$ so that 
 $(s^\dagger   ,q)\in D$ and $q\not\perp r$. Usual arguing will arrange
 that $q< r\restriction M$ and, loosely speaking, that, for
some easily chosen expansion $S_{r,q}$ of $S_r$, 
 $(\mathcal M_q\cup \mathcal M_r, S_{r,q}, m_r)$ will be a condition
 in $\mathcal P$ extending $q$.
 The   challenge is to ensure
 that such a condition is also an extension of $r$, which requires
that,
for each $\tilde M\in \mathcal M_q$, we have that
  $s\Vdash  e(y^{\tilde M}(s'\cap \tilde M)) \in \dot
  W_r(s' \cap M)$ for all $s,s'\in S_r$.
Some standard elementary submodels as side-conditions reasoning, together
with Lemma \ref{forproper} do the trick. 
We have applied Lemma 
 \ref{singleM} above so 
we also have that $s_i = \bar s_i\oplus  s_0$
   for each $i< \ell$. One thing we have gained is that in checking
 if $q\not\perp r$,  we need only check on membership in 
sets of the form $\dot W_r(\bar s_i\oplus  s_0 )$.

Let us say that $q$ is \underbar{\textbf{like}}  $ r$  (or $q\equiv r$)
 providing 
\begin{enumerate}
 \item $\mathcal M_r \cap M_0$ is an 
initial segment of $\mathcal M_q$, 
\item $\mathcal M_q \setminus \mathcal M_r = \{ M^q_i : i< \bar \ell\}$
has  cardinality $\bar \ell = |\mathcal M_r\setminus M|$,
 \item the tree structure $(S_q^\downarrow,<,\oplus)$ is
  isomorphic to $( S_r^\downarrow,<,\oplus)$,
  \item $ \{ s^q_i : i < \ell\} = \{ s\cap M^q_0 :  s\in S_q\}$ is also equal
  to $\{ \bar s_i \oplus  s^q_0 : i < \ell\}$,
\end{enumerate}

For $q\equiv r$,  and $k<\bar \ell$, 
let   $\langle s^{q,k}_i : i < \ell_{q,k}\rangle $ be
the set 
 $\{ s\cap M^q_k : s\in S_q\}$  ordered  lexicographically.
Also let  $\vec y_k\supy q $ denote the $\ell_{q,k}$-tuple
 $\langle y^{M^q_k} (s^{q,k}_i ) :  i<\ell_{q,k}\rangle$.
We have to do this because we want these to be members
of $Y^{\ell_{q,k}}$.
Thus we have the $\bar \ell$-tuple $\langle \vec y_k\supy q
 : k< \bar \ell\rangle$ associated
with each $q\equiv r$. 
Of course, $\ell_{q,k}$ is equal to $\ell_{r,k}$
 for $q\equiv r$.
Also, for each $k<\bar \ell$, let $i^\dagger_k$
denote the index $i$ with the property that $s^{r,k}_i < s^\dagger$.

\bigskip

Recursively define a collection of sets and names. 
First we have the $S$-name:
\[ \dot {\mathcal Y}_{\bar \ell} = 
 \{ (s, \langle \vec y_k\supy q : k< \bar \ell\rangle) : 
 (s,q)\in D \ \mbox{and}\ 
q\equiv r
\}~.\]

As usual, we have that $\dot {\mathcal Y}_{\bar \ell}\in M_0$
($q\equiv r$  can be described  within $M$). Now define, 
 for $k\in \{ \bar\ell-1, \bar \ell-2, \ldots, 0\}$ (in that order)
\begin{multline}
 \dot A(q,k) = \{ (\bar s, \vec y_k\supy {\bar q}) :
(\bar s, \langle \vec y_j\supy {\bar q} : j\leq k \rangle) \in
\dot {\mathcal 
  Y}_{k+1} \ , \\ 
 s^q_{i^\dagger_{k}} < \bar s\ \ \mbox{and}\ \ 
\langle \vec y_j\supy {\bar q} : j <   k\rangle =
\langle \vec y_j\supy q : j <  k\rangle   \}
\end{multline}
and let (for $k>0$)
\[ \dot {\mathcal Y}_{ k} = \{
 (s,\langle \vec y_m\supy q : m< k  \rangle) : s\forces \dot
 A(q,k) \in 
 \dot {\mathcal F}^+_{\langle s^{q,k}_i : i< \ell_{q,k}\rangle}\}~.\]

Thus $\dot A(q,\bar \ell -1)$ contains  the ``top'' element of the sequence
 $\langle \vec y_{k}\supy q : k< \bar \ell\rangle$.
Of course 
$s\forces \dot A(q,k) \in
 \dot {\mathcal F}^+_{\langle s^{q,k}_i : i< \ell_{q,k}\rangle}$
is equivalent to $(s,\langle s^{q,k}_i : i< \ell_{q,k}\rangle, 
\dot A(q,k))$ being a member of $ \mathcal A$. We use the notation
 $\dot A(q,k)$ rather than the more cumbersome
 $\dot A(\langle \vec y_j\supy q : j <  k\rangle )$ but let
us note that the definition depends only on the parameters
$\langle \vec y_j\supy q : j <  k\rangle $ and
 $\dot {\mathcal Y}_{k+1}$ in $M^q_k$
(and the latter is
an element of $M_0$).

Clearly $(s^\dagger, \langle \vec y_k\supy r :k <\bar \ell\rangle)\in 
 \dot {\mathcal Y}_{\bar \ell}$ and, for readability,
 let $k=\bar \ell-1$.
 Recall that $i= i^\dagger_{k}$
was  defined so that $s^r_{i}  = s^\dagger \cap
 M^r_{k}$. 
We  then have that
 $s^\dagger \Vdash \vec y_{k}\supy r \in \dot A(r,k)$. 
Now we
show that $(s^\dagger, \langle s^{r,k}_i : i < \ell_{r,k} \rangle, 
\dot A(r,k)\,)$ is in $\mathcal A$. First choose any $\xi_k\in 
 M^r_k \setminus M^r_{k-1}$ large enough so that all members
of $\{ s^{r,k}_i : i <\ell_{r,k}\}$ agree on the interval
 $[\xi_k, M^r_k\cap \omega_1)$. Let ${\bar s}^r_i = 
 s^r_i\restriction \beta_k$ for each $i<\ell_{r,k}$. 
As discussed
 above, we have that $\dot A(r,k)$ is an element of $M^r_k$.
By Lemma \ref{basiclemma}, there is a $\gamma\in M^r_k\cap \omega_1$
such
that each $s\in S_\gamma$ decides the statement
 $\dot A(r,k)\in \dot{\mathcal F}^+_{\langle {\bar s}^r_i :
   i<\ell_{r,k}\rangle} $. 
It follows from Lemma \ref{Melement} 
that $s^\dagger$ forces that $\vec y_k\supy r$ is a witness to the
fact that $\dot A\in \dot{\mathcal F}^+_{\langle {\bar s}^r_i :
   i<\ell_{r,k}\rangle}$. Then, by applying Lemma \ref{basiclemma}
 again, there is a $\beta_k\in M^r_k$ such that the tuple
 $( s^\dagger\restriction \beta_k, 
\langle {\bar s}^r_i :
   i<\ell_{r,k}\rangle, \dot A(r, k)\,)$ is in $\mathcal A\cap
   M^r_k$. This now shows that (for this value of $k$)
$(s^\dagger\restriction\beta_k , \langle
 \vec y_m\supy r : m <  k\rangle)$ is a member of $\dot
 {\mathcal Y}_{k}$.

Continuing this  standard argument, walking down from $s^\dagger$,
shows that, for each $k<\bar \ell$, there is a $\xi_k\in M^r_{k}$
such that
$(s^\dagger\restriction\xi_k , \langle
 \vec y_m\supy r : m <  k\rangle)$ is a member of $\dot
 {\mathcal Y}_{k+1}$.
Now we have that there is some $\beta_0\in M_0$ such that
$(s^\dagger\restriction\beta_0 , \emptyset)$ is a member of $\dot
 {\mathcal Y}_{1}$; and more importantly that 
 $(s^\dagger\restriction \beta_0, \emptyset, \dot A(r, 0 ))\in
 \mathcal A\cap M_0$. By Lemma \ref{forproper}, there is a 
 $\vec y_0\in M_0$ such that $s^\dagger\forces 
\vec y_0  \in \dot A(r,0  )$ and, for each 
 $s'\in S_r$ and $i<\ell_{r,0}$,
 $s'\forces e(~(\vec y_0)_i~ ) \in \dot W_r((
s_i \oplus s^\dagger)\restriction M_0)$. Now, we do not ``really'' mean
 $(\vec y_0)_i$ but rather  $y^{M^q_0}(s^0_i)$ for a suitable
$q_0\equiv r$ such that $\vec y_0 = \vec y_0\supy {q_0}$. If 
 we let $\gamma_0 = M^{q_0}_0\cap \omega_1$, then we have
that $s^{q_0}_{i_0^\dagger} = s^\dagger \restriction \gamma_0$, and
for each $i<\ell_{r,0}$, 
$$s^0_i = s^{q_0}_i =  \bar q_i \restriction \gamma_0 \oplus
s^{q_0}_{i^\dagger_0}~.$$

By elementarity, there is a $\alpha_1\in M_0$ such that
$$(s^\dagger\restriction\alpha_1, \langle 
 s^0_i : i < \ell_{r,0}\rangle, \dot A(q_0, 1))\in \mathcal A~.$$

We apply Lemma \ref{forproper} again and obtain 
 $r\equiv q_1 \in M_0$ such that
 $s^\dagger\forces 
 \vec y_1\supy {q_1}  \in \dot A(
q_0, 1 )$ and, for each 
 $s'\in S_r$ and $i<\ell_{r,1}$,
 $s'\forces e((\vec y_1)_i\,) \in 
\dot W_r((
s^{q_1}_i \oplus  s^\dagger)\restriction M_0)$.
 Unlike in the first step, it may
happen that $(s^{q_1}_i\oplus s^\dagger)\restriction M_0$
 is not in $S_r^\downarrow$, which poses no problem since
then $\dot W_r((
s^{q_1}_i \oplus  s^\dagger)\restriction M_0)$ 
is all of $\theta$.
 If 
 we let $\gamma_1 = M^{q_1}_1\cap \omega_1$, then we have
that $s^{q_1}_{i_1^\dagger} = s^\dagger \restriction \gamma_1$;
 and for each $i<\ell_{r,1}$~,  let
$s^1_i = s^{q_1}_i $.  So we again have
that $(\,s^\dagger, \langle s^1_i : i< \ell_{r_0} \rangle , \dot
A(q_1,2)\,)$ is in $\mathcal A$.

Well, we just repeat this argument for $\bar \ell$ steps until
 we find  $q = q_{\bar\ell} \equiv r$ 
with the property that $(s^\dagger, q)\in
 D$  and, for each $k<\bar\ell$, 
and for each 
 $s'\in S_r$ and $i<\ell_{r,k}$,
 $s'\forces e(
\vec y_k\supy q  ((s^q_{i}\oplus s^\dagger)\cap M^q_k)\,)
 \in \dot W_r((s^q_i \oplus
 s^\dagger)\restriction M_0)$. Again, for larger values of $k$, it may
 happen 
 that
 $(s^k_i \oplus s^\dagger) \cap M_0$ is not in $S_r^\downarrow$, and so
 $\dot W_r((s^k_i \oplus s^\dagger) \cap M_0)$ would simply equal $X$
or $\theta$.
\end{proof}

\section{On\label{nolarge}
 the consistency of $\mathbf{GA}$}

 Shelah has defined the
$\kappa $-p.i.c. (for ``proper isomorphism condition'') .  The reason for
this is that a  
countable support iteration of length 
$\omega_2$    of $\aleph_2$-p.i.c.proper posets will 
(under CH) satisfy the
$ \aleph_2$-chain condition, while just assuming that the 
factors themselves satisfy the
$\aleph_2$-chain condition does not guarantee the 
iteration
will. A diamond sequence on $\omega_2$  will help us decide
   which proper
$\aleph_2$-p.i.c.posets of size $\aleph_2$    to use in such an iteration.  The 
resulting iteration
will have cardinality $\aleph_2$    and the objects which we
 want 
to consider in the extension will also have cardinality $\aleph_2$. 
We have to
show that given 
any such reflected object there is an appropriate $\aleph_2$-p.i.c.
proper poset
 of cardinality $\aleph_2$    which will introduce the required set and 
that this set is still appropriate in the final model.
For this we use the method of Todorcevic \cite{matrices} in which side
conditions are finite sets (or matrices) of elementary submodels
rather than the more common method in which side conditions are
simple finite chains of elementary submodels. It is this change which
is the key in making the posets strongly $\aleph\sb2$-cc (and
iterable), thus removing the need for large cardinals to prove the
results. 

\begin{defn}[{\cite[Ch. VIII]{proper}}]
 \label{3.2}
  A poset $ P $ is said to satisfy the $ \kappa $-pic
 if whenever we have a sufficiently large cardinal $ \lambda  $,
 a well-ordering $\prec $ 
of $ H(\lambda )$, $i < j < \kappa $,
 two countable elementary submodels $ N_i $  and $N_j $
 of $\langle H(\lambda ), \prec , \in  \rangle $
 such that $ \kappa $  and $ P $ are in $ N_i  \cap  
N_j $ , $ i \in  N_i $ , $j \in  N_j$,
$N_i  \cap i = N_j \cap  j $,
 and suppose further that we are given $ p \in  
N_i $ and an 
isomorphism $ h : N_i \rightarrow N_j$
  such that $ h(i) = j $ and $ h $ is the identity on 
$N_i  \cap  N_j $  then there is a $ q \in  P $ such that : 
\begin{enumerate}
\item
$  q < p$ , $q < h(p) $ and $ q $ is both $ N_i $  and $ N_j $  generic , 
\item  if $ r \in  N_i  \cap  P $ and $ q' < q $ 
there is a $ q'' < q' $ so that \\
$q'' < r $  if and only if $  q'' < h(r) $ . 
\end{enumerate}
\end{defn}
 
\begin{proposition}    A countable\label{3.3} support iteration
 of length at most $ \omega_1$    of  
$\aleph_2$-p.i.c. proper posets is again $ \aleph_2$-p.i.c..
  Furthermore if  CH holds  and the 
iteration has length at most $\aleph_2\ $   then the iteration satisfies the  
$\aleph_2$-cc . 
\end{proposition}

\begin{proposition}
  A proper poset of cardinality $ \aleph_1 \ $   satisfies the 
$ \aleph_2$-pic. 
\end{proposition}

\begin{lemma}
  [{[CH]}]  If $ P $ is a\label{chaincondition} 
  proper $ \aleph_2$-p.i.c.
 poset and $ G $ is $P$-generic 
over $ V $ then $ V[G]\models \mathfrak c  = \omega_1 \ $   . 
\end{lemma}

Following \cite{matrices}, for a
countable
 elementary submodel $N$ of $H(\aleph_2)$, 
 we let $\overline{N}$ be the transitive collapse,
and we let $h_N : N \rightarrow \overline{N}$ be the collapsing map,
 i.e. $h_N(x) = \{ h_N(y) : y\in x\cap N \}$.

\begin{lemma}
Suppose that $N_1,N_2$ are countable\label{homega1}
 elementary submodels of $H(\aleph_2)$ such that
 $\overline{N_1} = \overline{N_2}$, and let
$h_{N_1,N_2} $ denote the map $h_{N_2}^{-1} \circ h_{N_1}$. Then 
$h_{N_1,N_2}$ is the identity on $H(\aleph_1)\cap N_1$ and
 for each $A\in N_1$ with $A\subset H(\aleph_1)$,
 $A\cap N_1  = h_{N_1,N_2}(A) \cap N_2$. 
\end{lemma}

\begin{proof}
It follows by $\in$-induction that  each
 $x\in N_1\cap H(\aleph_1)$,  $x\subset N_1$
and so $x\in \overline{N_1}$.  Therefore we also have, by
$\in$-induction, that $h_{N_1}(x) = x = h_{N_2}^{-1}(x)$.
\end{proof}

A family  $[\mathcal{  N} ] $  is
an elementary matrix if, for some integer $n>0$,
\begin{enumerate}
\item $[\mathcal{ N}]   = \{\mathcal  N_1,\ldots,\mathcal  N_n  \}$
\item for each $1\leq i\leq n$,  $\mathcal N_i$
is 
a  finite set of countable \ elementary \ submodels \ \ of
  $H(\omega_2 )$
\item for each $1\leq i\leq n$, $\overline{N_1} = \overline{N_2}$
for each pair $ N_1,N_2 \in  \mathcal N_i$,
\item for each $1\leq i< j\leq n$ and each 
 $ N_i  \in   \mathcal N_i$, there is an
  $N_j \in  \mathcal N_j$ with $N_i\in N_j $.
\end{enumerate}

  It will be convenient to let  $N \in   [\mathcal{ N}]$, 
for an elementary submodel  $N$  of $H(\aleph_2)$,
be an abbreviation for  $ N \in  \mathcal N $
for some   $\mathcal  N \in  [\mathcal N]$.

\begin{lemma}[CH] If $\dot{\mathcal I}$ is an\label{p22pic}
 $S$-name of a  $P$-ideal on $\omega_1$ such
that 1 forces  that $\dot{ \mathcal I}\cap [E]^{\aleph_0}$ is not empty for
all stationary sets $E\subset\omega_1$, then 
there is an
$S$-preserving $\aleph_2$-p.i.c.proper poset $\mathcal P$ of cardinality
$2^{\omega_1} $ such that $\mathcal P$ forces that
 there is an 
 $S$-name $\dot E$ of a stationary set
with $1\Vdash [\dot E]^{\aleph_0}\subset \dot
 {\mathcal I}$.
\end{lemma}

\begin{lemma}[CH] If $\dot X$ is an $S$-name of a\label{ppipic}
 sequentially compact
  non-compact first countable space, then
there is an
$S$-preserving $\aleph_2$-p.i.c. proper poset $\mathcal P$ of cardinality
$2^{\aleph_1} $ such that $\mathcal P$ forces that
there is an $S$-name 
$\{ \dot x_\gamma : \gamma\in \omega_1\}\subset \dot X$ 
that is forced to contain an uncountable free sequence,
and, if $\dot X$ is first countable,  to
be a homeomorphic copy of $\omega_1$. 
\end{lemma}

The proofs are very similar with the same underlying idea in that we
replace elementary chains from the original proofs with
 elementary matrices. The usage of elementary
matrices is the device to make the poset satisfy the
$\aleph_2$-p.i.c. The proof that the modified poset is proper and
$S$-preserving relies on the fact that CH guarantees that
the key combinatorics take
place within $H(\aleph_1)$ and so, by Lemma \ref{homega1}
 no new arguments or constructions
are required.
Since it is newer, we sketch the
proof of Lemma \ref{ppipic} and leave the proof of Lemma \ref{p22pic}
to the interested reader.  In actual fact, this method isn't really
needed for the consistency of $\mathbf{P}_{22}$ because the 
needed poset can be chosen to have cardinality $\mathfrak c$.
The reason this is not true for $\mathbf{PPI}^+$ is that we must
utilize the construction of the maximal filter of $S$-sequentially closed sets
which may have cardinality $2^{\omega_1}$.
We simply indicate the modifications needed
to the proof of Lemma \ref{ppithm}.

\begin{proof}
Let $\dot X$ be the $S$-name as formulated in the Lemma. 
By Lemma \ref{itsomega1}, we can pass to a subspace and assume 
that either $\dot X$ is separable or that the set
$\omega_1\times\{0\}$ is   a subset and is
forced
to not have a complete accumulation point. Since we are assuming CH,
we can, in either case, pass to an $S$-name of
cardinality $\aleph_1$ for a  subspace that is still sequentially
compact and not compact.  With this space having cardinality
$\aleph_1$ it is clear that also in the separable case, we can assume
that $\omega_1\times \{0\}$ is a subspace with no complete accumulation point. 
We can now assume that the base set for $\dot X$ is 
$\omega_1\times \omega_1$ (i.e. any set that is a
subset of $H(\aleph_1)$.
The entire topology $\dot \tau$ for $\dot X$ can be coded as a subset
of $S\times \omega\times \omega_1^4$ where
$(s,m,\alpha,\beta,\gamma,\delta)\in \dot \tau$ codes the fact
that $s$ forces that $(\gamma,\delta)$ is in $\dot U(~(\alpha,\beta)~,
m)$. The family $\bigcup \{ Y_\alpha : \alpha\in \omega_1\}$ and
  $\mathbf{WF}$ 
as defined in \S 3.1  are already subsets of $H(\aleph_1)$. 
We also fix a well-order $\prec_{\omega_1}$ of $H(\aleph_1)$.

Finally, with no changes, the family $\mathcal A$ as defined in 
 Definition \ref{needA} is a subset of $H(\aleph_2)$. This family
 $\mathcal A$ was the key parameter in defining our poset $\mathcal
 P$. Lemma \ref{Melement} holds for any $M\prec
 (H(\aleph_2),\tau, \mathcal A)$ (meaning $\tau\in M$ and $\mathcal A$
 is a new term in the language).
 The choice of the sequence $\{ y^M(s) : s\in S_{M \cap \omega_1}\}$
 from Lemma \ref{Melement} will be the $\prec_{\omega_1}$-minimal 
 such sequence.

\begin{claim}
Consider any set $\mathcal N$
 of pairwise isomorphic countable elementary submodels of 
 $(H(\aleph_2), \dot \tau, \mathcal A)$; i.e.
  $\overline{N} = \overline{N'}$ for $N,N'\in \mathcal N$.  Let $\delta = N\cap
  \omega_1$ for any $N\in \mathcal N$.\label{claimsame}
  Let $N_1, N_2$ be elements of $\mathcal N$.
  We then have that the two sequences
    $\langle y^{N_1}(s) : s \in S_\delta\rangle$ and 
    $\langle y^{N_2}(s) : s \in S_\delta\rangle$
    are the same.
 \end{claim}
 
 \bgroup
 \def\proofname{Proof of Claim \ref{claimsame}:}
 \begin{proof}
To prove the claim,  let $(\bar s,\{s_i : i < n \}, \dot A)$ be any member of
   $\mathcal A\cap N_1$ and assume that $\bar s < s\in S_\delta$.
    Choose $B\subset Y^n \cap N_1$  such that $s\Vdash
     B\subset \dot A$ and 
      $s \Vdash \langle y^{N_1}(s\oplus s_i)  : i < n\rangle \in 
       B^{(\delta+1)}$.  By Lemma \ref{homega1}, 
   $h_{N_1,N_2}((\bar s, \{ s_i : i< n\}, \dot A))$ is in $\mathcal A
   \cap N_2$. Since $\dot A\subset H(\aleph_1)$, we also have
   by Lemma \ref{homega1}, that $h_{N_1,N_2}(\dot A)\cap N_2$
   is equal to $\dot A \cap N_1$. Therefore, we have
   that $s$ also forces that $B$ is a subset of 
    $h_{N_1,N_2}(\dot A)$. Well this shows that
      $\langle y^{N_1} (s\oplus s_i) : i < n \rangle$ satisfies
   this particular  requirement of
      $\langle y^{N_2} (s\oplus s_i) : i < n \rangle$
   with respect to $h_{N_1,N_2}((\bar s, \{ s_i : i< n\}, \dot A))$.  
   Since $h_{N_1,N_2}$ is an isomorphism, this shows
   that $\langle y^{N_1}(s) : s\in S_\delta\rangle$ works as a 
   choice for 
 $\langle y^{N_2}(s) : s\in S_\delta\rangle$, and so, indeed,
 they are the same.
 \end{proof}
 
 \egroup

 \bigskip

 A condition $p\in \mathcal P$ consists of $([{\mathcal N}_p], S_p, m_p)$
 where $[\mathcal N_p]$ is a 
   elementary matrix of
submodels of 
 $(H(\aleph_2), \prec_{\omega_1},\dot \tau, \mathcal A)$. 
 We let    
 $\delta_p = N \cap \omega_1$
for any maximal $N\in \mathcal N_p$. We require that 
$m_p$ is a positive integer and
 $S_p$ is a finite subset of $S_{\delta_p}$. 
  
  For each $s\in S_p$ and each 
  non-maximal $N\in [{\mathcal N}_p] $
   we define an $S$-name
   $\dot W_p(s\restriction N)$ of a neighborhood of $e(y^N(s\restriction N))$.
   It is defined as the name of the
 intersection of all sets of the form $\dot U(s'\restriction N',m_p)$ 
   where $s'\in S_p$, non-maximal $N'\in [\mathcal M_p] $, and
$s\restriction N\subset s'\restriction N'$
    and $e(
    y^N(s\restriction N)) \in \dot U(s'\restriction N')$. We adopt
   the convention that $\dot W_p(s\cap N)$ is all of $X$ if
 $s\cap N\notin S_p^\downarrow$.
  \medskip
  
  The definition of $p< q$ is that 
 each  $N\in [\mathcal N_q]$ is a member of $[ \mathcal N_p]$,
  $m_q\leq m_p$,  
   $S_q \subset S_p^\downarrow$ 
and   for
   each $s'\in S_p$ 
   and 
  $s\in S_q$,  
   we have that $s'$ forces that
   $e( y^N(s\restriction N))\in \dot W_q(s\restriction N')$
   whenever $N\in [\mathcal N_p] $ and $N\notin  [ \mathcal N_q]$
and   $N'$ is a minimal member of 
$[\mathcal N_q]\setminus N $, which is itself not a maximal member
of $[\mathcal N_q]$.
Again we note 
that we make no requirements on sets of
the form $\dot U(s, m_q)$ for $s\in S_q$.
 
 Because of Claim \ref{claimsame}, the proof that $\mathcal P$ is proper
 and $S$-preserving proceeds exactly as in the proof of 
 Theorem \ref{ppithm}.  
 
 \begin{claim}
  $\mathcal P$ satisfies\label{itspic} 
 the $\aleph_2$-p.i.c. 
 \end{claim}
 
 \bgroup
 \def\proofname{Proof of Claim \ref{itspic}:}

\begin{proof}
 Let $\lambda$ be a
 sufficiently large cardinal, fix 
 a well-ordering $\prec $ 
of $ H(\lambda )$, let $i < j < \omega_2 $ be such
that there are
 two countable elementary submodels $ N_i $  and $N_j $
 of $\langle H(\lambda ), \prec , \in  \rangle $
 such that   $\mathcal P $  is in $ N_i  \cap  
N_j $ , $ i \in  N_i $ , $j \in  N_j$,
$N_i  \cap i = N_j \cap  j $,
 and suppose further that we are given $ p \in  \mathcal P\cap 
N_i $ and an 
isomorphism $ h : N_i \rightarrow N_j$
  such that $ h(i) = j $ and $ h $ is the identity on 
$N_i  \cap  N_j $.  We must show that
  there is a $ q \in  P $ such that : 
\begin{enumerate}
\item
$  q < p$ , $q < h(p) $ and $ q $ is both $ N_i $  and $ N_j $  generic , 
\item  if $ r \in  N_i  \cap  P $ and $ q' < q $ 
there is a $ q'' < q' $ so that \\
$q'' < r $  if and only if $  q'' < h(r) $ . 
\end{enumerate}

We first show that since $\mathcal P\in N_i\cap N_j$, we also 
have that $\{\prec_{\omega_1},\dot \tau, \mathcal A\}\in N_i\cap N_j$. 
The reason is that the collection $\{ N : (\exists p\in \mathcal P) 
N\in [\mathcal N_p]\}$ is in $N_i\cap N_j$. 
It follows that $N_i' = (N_i\cap H(\aleph_2), \prec_{\omega_1},
 \dot \tau , \mathcal A )$ is an elementary submodel of
  $(H(\aleph_2), \prec_{\omega_1}, \dot \tau, \mathcal A)$. 
  $N_j'$ defined similarly is as well.  The definition of 
  the  $[\mathcal N_q]$ for $q$ is canonical. Given
  that $[\mathcal N_p] = \{ \mathcal N_1, \mathcal N_2,
  \ldots , \mathcal N_n\}$, we set
   $[\mathcal N_q ] = \{ \mathcal N_1\cup h(\mathcal N_1), 
    \ldots, \mathcal N_n \cup h(\mathcal N_n) ,
       \{ N_i', N_j'\}\} $.   The existence of $h$ ensures
       that $\overline{N_i'} = \overline{N_j'}$.  Since
        $[\mathcal N_p]\in N_i'$ and $
       h([\mathcal N_p])= [\mathcal N_{h(p)}]\in N_j'$
     we have that $[\mathcal N_q]$ is an elementary matrix. 
Choose $S_q \subset S_{N_i\cap \omega_1}$ to be any finite
set such that $S_p= h(S_p) \subset S_q^\downarrow$.  
We already know that $q$ is both $N_i$ and $N_j$ generic
from the arguments in Theorem \ref{ppithm}. 

Finally let $r\in N_i\cap \mathcal P$ and $q'<q$ with $q'\in \mathcal P$.
We may assume, by symmetry, that there is a $q'' <q'$ that is 
also below $r$.  Let $[\mathcal N_{q''} ] $ be listed
as $\{ \mathcal N_1, \ldots , \mathcal N_k \}$ and let
 $1<\ell\leq k$ be chosen so that $ N_i' \in \mathcal N_\ell$.
 For $1\leq m< \ell$, let $\mathcal N_m^i = \mathcal N_m\cap N_i$.
 Of course we have that $N\in N_j'$ for each 
 $1\leq i<\ell$ and each $N\in h(\mathcal N_m^i)$.
 It is the easily verified that 
  $$[\mathcal N_{\tilde q} ] =
  \{ h(\mathcal N_1^i)\cup \mathcal N_1, \ldots ,
  h(\mathcal N_{\ell-1}^i)\cup \mathcal N_{\ell-1},
  \mathcal N_{\ell},\ldots, \mathcal N_k \}$$
  is an elementary matrix,
  and so
  $\tilde q \in \mathcal P$ where
  $( [\mathcal N_{\tilde q}], S_{q''}, m_{q''})$, where
  $$[\mathcal N_{\tilde q} ] =
  \{ h(\mathcal N_1^i)\cup \mathcal N_1, \ldots ,
  h(\mathcal N_{\ell-1}^i)\cup \mathcal N_{\ell-1},
  \mathcal N_{\ell},\ldots, \mathcal N_k \}~~.$$
  
It is immediate that $\tilde q < q''$, and so $\tilde q < q',r$.
We just have to show that $\tilde q$ is also below $h(r)$.
Since $q'' < r$, we have that $[\mathcal N_r]\in N_i$ is a submatrix
of $\{ N_i\cap \mathcal N_1,\ldots, N_i\cap \mathcal N_{\ell-1}\}$,
and so $[\mathcal N_{h(r)}]$ is a submatrix of
$[\mathcal N_{\tilde q}]$.
\end{proof}
 
 \egroup

This completes the proof.
\end{proof}

\begin{definition} The stationary set of ordinals $\lambda \in \omega_2$
with uncountable cofinality is denoted
as $S^2_1$.  
The principle $\diamondsuit(S^2_1)$ is the statement:

There is a family $\{ X_\lambda : \lambda \in S^2_1\}$ such that
\begin{enumerate}
\item for each $\lambda\in S^2_1$, $X_\lambda\subset\lambda$,
\item for each $X\subset\omega_2$, the set  $E_X = \{ \lambda \in 
 S^2_1 : X\cap \lambda = X_\lambda\}$ is stationary.
\end{enumerate}
\end{definition}

\begin{theorem}
Assume $2^{\aleph_0}=\aleph_1$ and $\diamondsuit(S^2_1)$.
There is a proper poset $\mathbb P$ so that in the forcing extension
by $\mathbb P$ there is a coherent Souslin tree $S$ such that,
in the full forcing extension by $\mathbb P * S$, 
the statement $\mathbf{GA}$ holds.
 \end{theorem}

\begin{proof}
We construct a countable support iteration sequence 
 $\langle \mathbb P_\alpha , \dot{\mathbb Q}_\beta :
  \alpha \leq \omega_2, \beta < \omega_2\rangle$.
  By induction, we assume that $\mathbb P_\alpha$ 
  is proper, has cardinality at most $\aleph_2$, and that 
  $$\Vdash_{\mathbb P_\alpha} \dot{\mathbb Q}_\alpha
  \ \ \mbox{satisfies the}\ \ 
 \aleph_2\mbox{-p.i.c}\ \ .$$
 
Note that by Lemmas \ref{3.3} and \ref{chaincondition} we will
have that, for each $\alpha < \omega_2$,   CH holds
in the forcing extension by $\mathbb P_{\alpha}$.
  We may assume that $\dot {\mathbb Q}_0$, and
  therefore $\mathbb P_1$ is constructed so that
  there is a $\mathbb P_1$-name, $\dot S$ of a coherent Souslin
  tree (henceforth we suppress the dot on the $S$).
  We further demand of our induction that, for $\alpha\geq 1$
  $$\Vdash_{\mathbb P_\alpha} \dot{\mathbb Q}_\alpha
  \ \ \mbox{is}\ \ S\mbox{-preserving}~.$$
 
For each ordinal $0<\alpha \in \omega_2\setminus S^2_1$, we 
let $\dot{\mathbb Q}_\alpha$ denote the $\mathbb P_\alpha$-name
of the standard Hechler poset for adding a dominating real. This ensures
that $\mathfrak b = \omega_2$ in the forcing extension by $\mathbb P_{\omega_2}$. 

For the rest of the construction, fix any function $h$ from $\omega_2$
onto $H(\aleph_2)$.
  Also let $\{ X_\lambda : \lambda \in S^2_1\}$ be 
a $\diamondsuit(S^2_1)$-sequence.

Now consider $\lambda\in S^2_1$ and let
 $x_\lambda = h[ X_\lambda]$.  We define $\dot{\mathbb Q}_\lambda$ according to
 cases:
 \begin{enumerate}
 \item if $x_\lambda$ is the $\mathbb P_\lambda*S$-name of a P-ideal on $\omega_1$
 such that $$1\Vdash [E]^{\aleph_0}\cap x_\lambda \ \ \mbox{is not empty for all stationary sets \ } E\subset\omega_1$$
 then $\dot{\mathbb Q}_\lambda$ is the $\mathbb P_\lambda$-name of the poset
 from Theorem \ref{p22pic},
 \item if $x_\lambda$ is the $\mathbb P_\lambda*S$-name of a
 subset of $\lambda\times\lambda\times\lambda  $ so that 
 if we define, for $\xi,\eta\in \lambda$,
  $\dot U(\xi,\eta)  $ to be the $\mathbb P_\lambda*S$-name of
  the subset of $\lambda$   such that $$\{ (\xi,\eta)\}\times
  \dot U(\xi,\eta) = x_\lambda \cap (\{(\xi,\eta)\}\times \lambda )\ \ ,$$
  i.e.  for $( (p,s), (\xi,\eta,\gamma) )$ in the set $x_\lambda$, 
      $( (p,s), \gamma )$ is in the name $\dot U(\xi,\eta)$,
  and $\mathbb P_\lambda* S$
   forces that the family $\{ \dot U(\xi,\eta) : \eta\in \lambda\}$ is a
  local base for $\xi$ in a 
 sequentially
 compact  regular topology on $\lambda$, and that no finite subset of
  $\{ \dot U(\xi,\eta) : \xi,\eta\in \lambda\}$ covers $\lambda$,
  then $\dot {\mathbb Q}_\lambda$ is the $\mathbb P_\lambda$-name of
  the poset from Theorem \ref{ppipic}.
      \item in all other cases, $\dot {\mathbb Q}_\lambda$ is the
      $\mathbb P_\lambda$-name of the
       Cohen poset $2^{<\omega}$.
 \end{enumerate}
 
 Assume that $\dot{\mathcal I}$ is a $\mathbb P_{\omega_2}*S$-name
 of a P-ideal on $\omega_1$ satisfying that there is some
  $(p_0,s_0)\in \mathbb P_{\omega_2}*S$ forcing that
  $[E]^{\aleph_0}\cap \dot{\mathcal I}$ is not empty for all stationary
  sets $E\subset \omega_1$.
The ideal of all countable subsets of $\omega_1$ is also such an
ideal, so we can find a $\mathbb P_{\omega_2}*S$-name
 $\dot{\mathcal J}$ such that $(p,s) \Vdash \dot{\mathcal J} 
  = \dot{\mathcal I}$, and $1$ forces that 
  $[E]^{\aleph_0}\cap \dot{\mathcal I}$ is not empty for all stationary
  sets $E\subset \omega_1$.   
  Let $X\subset \omega_2 $
   be chosen to be the set of all   $\xi\in \omega_2$
   with the property that 
  there is a $\mu<\omega_2$ such that
  $h(\xi)$ is a  $\mathbb P_\mu*S$-name  
  with 
  $1\Vdash h(\xi)\in \dot{\mathcal J}$.  
  There is  a cub $C\subset\omega_2$ such that for each 
  $\mu<\mu' \in C$:
  \begin{enumerate}
  \item   
  the collection $\{ h(\xi) : \xi \in X\cap \mu\}$ 
  is a collection of $\mathbb P_{\mu'}*S$-names, 
  \item for each countable subset
 $\{\xi_n : n\in \omega\}$  of   $X\cap \mu  $ there
 is a $\xi<\mu'$ such that 1 forces that $h(\xi_n)$
  is almost contained $h(\xi)$ for
  each $n$,
  \item  
 every $\mathbb P_\mu*S$-name 
  that is forced by 1 to be a member of $\dot{\mathcal J}$ 
  is equivalent to a name in $\{ h(\xi) : \xi \in X\cap \mu'\}$.
  \end{enumerate}

  Therefore, there is a $\lambda\in E_X\cap C$ such
  that $X_\lambda = X\cap \lambda$. We may of course
  assume that $p_0\in \mathbb P_\lambda$.   
  Routine checking now shows that $x_\lambda$ satisfies
   clause (1) in the definition of $\dot{\mathbb Q}_\lambda$.
   It follows that $\mathbb P_{\omega_2}*S$ is a model
   of $\mathbf{P}_{22}$.
   
   Now suppose that we have a $\mathbb P_{\omega_2}*S$-name
   of a sequentially compact non-compact space. We note
   that $\mathbb P_{\omega_2}*S$ forces that $2^{\aleph_0} = \aleph_2$.
   Therefore, by Lemma \ref{itsomega1}, we can pass to a name of a  
   sequentially
   compact non-compact subspace which has cardinality 
   at most $\omega_2$. In fact, we can assume this space
   has cardinality exactly $\omega_2$ by taking the free union
   with the Cantor space.  
Let $\dot Z$ denote
    the $\mathbb P_{\omega_2}*S$-name of this space.
   Again, by Lemma \ref{itsomega1}
   we can assume that each point of the space has a 
   separable neighborhood. This means that, with re-indexing,
we can  assume
   that the base set for the space is the ordinal
    $\omega_2$ and 
  $\{ \dot U(\xi,\eta) : \xi ,\eta\in \omega_2\}$ is
    the list of $\mathbb P_{\omega_2}*S$-names of the neighborhood
    bases of the points, and that no finite subcollection covers.
   We  define $X$ to be the set of all those
    $\alpha\in \omega_2$ such that $h(\alpha)$ is 
    a tuple of the form  $( (p,s) , (\xi,\eta,\gamma) )$, 
i.e.  a $\mathbb P_{\omega_2}*S$-name of a member of
 $\omega_2\times\omega_2\times \omega_2$,   
     where $(p,s) \Vdash \gamma \in \dot U(\xi,\eta)$.
     
     We again want to choose a $\lambda$ in $E_X \cap C$
     for some special cub set $C$ and in this case it is much
     simpler to make use of uncountable elementary submodels. 
     Let $\kappa$ be any regular cardinal greater than $2^{\omega_2}$,
     and let $\{ M_\alpha : \alpha \in \omega_2\}$ be chosen so that,
     for each $\alpha\in \omega_2$:
     \begin{enumerate}
     \item $X, h$ and $\mathbb P_{\omega_2}*S$ are in $ M_\alpha$
     \item $\omega_1\subset M_\alpha$ and $M_\alpha$ has cardinality $\aleph_1$,
     \item for each $\beta<\alpha$, every countable subset of $M_\beta$
      is an element of $M_\alpha$,
      \item $M_\alpha$ is an elementary submodel of $H(\kappa)$,
      \item if $\alpha$ is a limit ordinal, then $M_\alpha = \bigcup\{
       M_\beta : \beta < \alpha\}$.
     \end{enumerate}
     
     Items (2) and (4) guarantee that $M_\alpha\cap \omega_2$ is 
     an initial segment of $\omega_2$ -- hence an ordinal.
The chain $\{ M_\alpha : \alpha \in \omega_2\}$ is a continuous chain
because of item (5),  and so $C = \{ M_\alpha \cap \omega_2 : \alpha
\in \omega_2\}$ is a closed and unbounded subset of $\omega_2$.
Now we choose $\lambda \in E_X\cap C$, we can also choose 
$\lambda$ so that it is $M_\lambda$ with $M_\lambda\cap \omega_2 $
being $\lambda$. 
  Using items (1), (3) and (4) and the fact that $\lambda \in S^2_1$,
   it is now easy to show that 
  $x_\lambda$ will satisfy the requirement (2) in the construction
  of $\dot{\mathbb Q}_\lambda$. It then follows,
  as in the proof of Lemma \ref{3.9},
   that
  $\mathbb P_{\lambda+1}*S$ will force the existence of the necessary
   $\omega_1$-sequence showing that $\dot Z$ is not a counterexample
   to $\mathbf{PPI}^+$.
\end{proof}


\def\cprime{$'$}

\end{document}